\documentclass[11pt]{article}
\usepackage{authblk} 
\usepackage[margin=1in]{geometry}  
 
\usepackage{amsmath,amsthm,amssymb}
\usepackage{booktabs}
\usepackage{multirow}
\usepackage{threeparttable}
\newtheorem{theorem}{Theorem}[section]
\newtheorem{lemma}{Lemma}[section]
\newtheorem{corollary}{Corollary}[section]
\newtheorem{remark}{Remark}[section]
\newtheorem{definition}{Definition}[section]

\newtheorem{problem}{Problem}[section]

\newtheorem{proposition}{Proposition}[section]

\usepackage{natbib}

\usepackage{setspace}              

\usepackage{pgfplots}

\usepackage{color,hyperref}
\definecolor{darkblue}{rgb}{0.0,0.0,0.6}
\hypersetup{colorlinks,breaklinks,linkcolor=darkblue,urlcolor=darkblue,anchorcolor=darkblue,citecolor=darkblue}
\usepackage{algpseudocode}
\usepackage{algorithm}

\usepackage[bb=dsserif]{mathalpha}
\usepackage{bm}

\title{
	\bf Is Noisy Data a Blessing in Disguise? A Distributionally Robust Optimization Perspective\thanks{
		This manuscript is a preprint and is currently under review for publication. Chung-Han Hsieh (\href{mailto:ch.hsieh@mx.nthu.edu.tw}{ch.hsieh@mx.nthu.edu.tw}) and Rong Gan (\href{mailto:GanZong@gmail.com}{GanZong@gmail.com}) are both with the Department of Quantitative Finance, National Tsing Hua University, Hsinchu, Taiwan. This work was partly supported by the National Science and Technology
		Council (NSTC) under Grants: NSTC113--2628--E--007--015-- and NSTC114--2628--E--007--006--.
		} 
}
 
\author[1]{Chung-Han Hsieh}
\author[1]{Rong Gan}

\affil[1]{\small Department of Quantitative Finance, National Tsing Hua University, Hsinchu, Taiwan, 30004. }

\date{}

\begin{document}
	
	\maketitle
	
	\begin{abstract}
		Noisy data are often viewed as a challenge for decision-making. This paper studies a distributionally robust optimization (DRO) that shows how such noise can be systematically incorporated.  
		Rather than applying DRO to the noisy empirical distribution, we construct ambiguity sets over the \emph{latent} distribution by centering a Wasserstein ball at the noisy empirical distribution in the observation space and taking its inverse image through a known noise kernel. 
		We validate this inverse-image construction by deriving a tractable convex reformulation and establishing rigorous statistical guarantees, including finite-sample performance and asymptotic consistency.
		Crucially, we demonstrate that, under mild conditions, noisy data may be a ``blessing in disguise." Our noisy-data DRO model is less conservative than its direct counterpart, leading to provably higher optimal values and a lower price of ambiguity. 
		In the context of fair resource allocation problems, we demonstrate that this robust approach can induce solutions that are structurally more equitable. 
		Our findings suggest that managers can leverage uncertainty by harnessing noise as a source of robustness rather than treating it as an obstacle, producing more robust and strategically balanced decisions.
	\end{abstract}

\providecommand{\keywords}[1]
{
	\small	
	\textbf{\textit{Keywords---}} #1
}

\maketitle

\section{Introduction}

Let $X$ be a random vector governed by an unknown, \textit{latent} distribution $\mathbb{F}$. The goal of a stochastic optimization problem is to maximize expected utility with respect to  $\mathbb{F}$: $\sup_{w \in \mathcal{W}} \mathbb{E}^{\mathbb{F}} [ U(w, X)]$, where~$w$ is a decision variable chosen from a feasible set $\mathcal{W}$, and~$U(\cdot)$ is a concave and nondecreasing utility function. 
A key challenge of the stochastic optimization problem is that the latent distribution $\mathbb{F}$ is typically unobservable and must be inferred from the data. 
Ideally, one would have access to independent and identically distributed (i.i.d.) samples drawn directly from the latent distribution $\mathbb{F}$,  
However, in many practical scenarios, data are inevitably corrupted by a noise process that blurs, distorts, or degrades the underlying information; e.g., see \cite{farokhi2023distributionally, van2024efficient, cai2025distributionally}.

While stochastic optimization under uncertainty has received considerable attention, the challenges posed by noisy data remain less studied relative to the noiseless-sample setting, and ignoring noise can lead to seriously misleading decisions.
In finance, for example, true signals are often inherently obscured by noise, making it challenging for investors to estimate expected returns for stocks or portfolios, see, e.g.,~\cite{black1986noise, banerjee2015signal}.
In other domains, data can be intentionally disturbed by noise from a known distribution to preserve privacy, see~\cite{muralidhar1999general, dwork2006our}, thereby further increasing uncertainty and complicating decision making.
Under these circumstances, where the latent distribution~$\mathbb{F}$ is unknown and the observed data are further corrupted by noise, making reliable decisions becomes even more difficult.

To cope with limited knowledge of $\mathbb{F}$, a common approach is to approximate it using the empirical distribution, denoted by $\widehat{\mathbb{F}}$.
However, even in the ideal case when noiseless data from $\mathbb{F}$ is available to construct $\widehat{\mathbb{F}}$, the resulting optimal decision can still perform poorly out-of-sample due to overfitting, the phenomenon known as the \textit{optimizer's curse}, see~\cite{smith2006optimizer, gao2023finite}, and the situation may become even worse when observations are corrupted by noise.
To address this issue, an emerging trend is to incorporate data uncertainty from real-world applications into decision-making. 
Early influential work, such as \cite{bertsimas2004price, ben2009robustOptimization, ben2013robust}, proposed a robust framework that enhances decision robustness while allowing control over~conservatism.

A broader framework addressing this uncertainty is distributionally robust optimization (DRO), e.g., see \cite{wiesemann2014distributionally}.
DRO accounts for uncertainty in the underlying distribution by considering an \textit{ambiguity set} of distribution $\mathbb{F}$, denoted by $\mathcal{P}$, see~\cite{rahimian2022frameworks, kuhn2025distributionally} for a comprehensive survey.
Rather than relying on a single estimated distribution as in the original stochastic optimization, DRO seeks decisions that maximize the expected utility under the worst-case distribution within the ambiguity set $\mathcal{P}$. 
That is, DRO reformulates the stochastic optimization into a max-min framework: $\sup_{w \in \mathcal{W}} \inf_{\mathbb{F} \in \mathcal{P}} \mathbb{E}^{\mathbb{F}} [ U(w, X)]$.
Interestingly, in the presence of noise, we will demonstrate that DRO can somehow turn the noise itself into an advantage, leading to higher optimal values and, consequently, a lower price of ambiguity.

DRO models critically depend on how ambiguity sets are defined to capture distributional uncertainty. We first briefly review common types. Ambiguity sets are typically classified as \textit{discrepancy-based} or \textit{moment-based} \cite{rahimian2022frameworks}.
Discrepancy-based ambiguity sets consider a neighborhood around a nominal distribution using a discrepancy measure, including: 
$(i)$ Optimal transport discrepancy such as the Wasserstein metric \cite{mohajerin2018data, blanchet2022distributionally, gao2023distributionally, li2023wasserstein, zhang2024short}, 
$(ii)$ $\phi$-divergences like Kullback-Leibler divergence \cite{hu2013kullback} and $\chi^2$-divergence \cite{levy2020large};
$(iii)$ Total variation distance \cite{jiang2018risk, farokhi2023distributionally};
and $(iv)$ Polyhedral set~\cite{hsieh2024solving}.
On the other hand, moment-based ambiguity sets include all distributions whose moments satisfy certain properties \cite{li2013portfolio}.

\begin{figure}[htbp]
\centering
\includegraphics[width=0.5\linewidth]{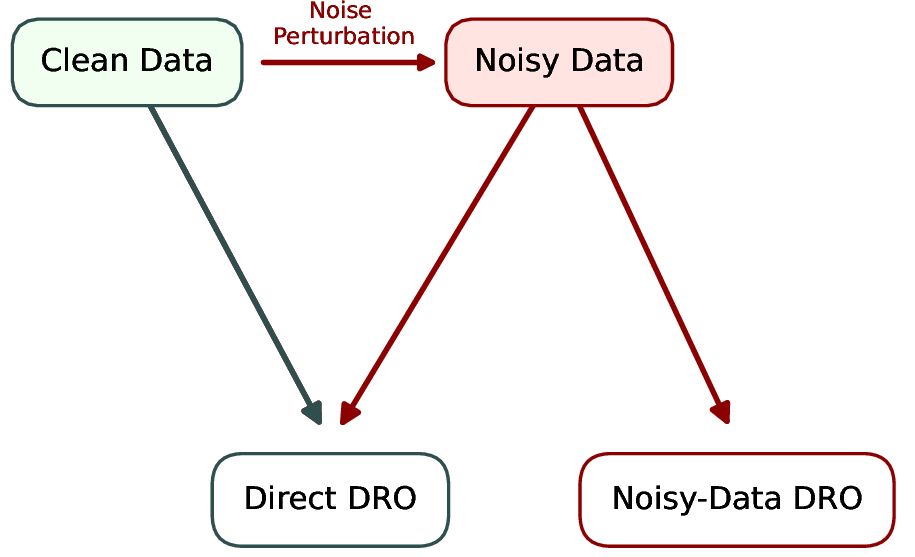}
\caption{Conceptual Framework for Direct and Noisy-Data DRO.}
\label{fig: dro_framework}
\end{figure}

	
	
	


While the DRO approach is designed to handle distributional ambiguity, it generally focuses on distributional uncertainty rather than noise in the observed data. 
Most data-driven DRO work assumes access to noiseless data, typically requiring~i.i.d. samples from~$\mathbb{F}$, e.g., \cite{mohajerin2018data, li2023wasserstein, hsieh2024cost, zhang2025short}. 
In practice, however, one often observes only the corrupted empirical distributions.
Contaminated data is typically addressed with robust statistics, see \cite{blanchet2024distributionally}, which deals with outliers or measurement errors during the pre-decision stage. However, DRO for noisy data has been less explored.
To address this gap in the DRO framework, we propose a DRO approach that builds ambiguity sets over the latent distribution, obtained as the inverse image of a noisy distribution centered at the noisy empirical distribution through a known noise kernel.
For clarity, we illustrate the conceptual setting of DRO construction in Figure~\ref{fig: dro_framework}. 
The idealized (clean) world assumes access to clean data, i.e., the empirical distribution~$\widehat{\mathbb{F}}$.
In contrast, the realistic (noisy) world reflects practical data collection scenarios, where one typically observes only a corrupted distribution~$\widehat{\mathbb{F}}^\star$, resulting from~$\widehat{\mathbb{F}}$ being perturbed by a known noise kernel $\mathbb{O}$.
In this paper, we investigate two modeling approaches:
$(i)$ Direct DRO, which treats the data as if it were clean---regardless of whether it is clean or noisy---by directly applying the standard DRO formulation without accounting for noise;
$(ii)$ Noisy-Data DRO, which explicitly acknowledges the presence of noise and incorporates it into the DRO formulation. This is the central focus of our work.

Due to the presence of noise, even with a sufficiently large number of samples, we can only construct the noisy empirical distribution $\widehat{\mathbb{F}}^\star$. 
The noisy empirical distribution, however, converges to the corrupted distribution $\mathbb{F}^\star$---the noisy counterpart of the latent (clean) distribution $\mathbb{F}$, corrupted via a noise distribution $\mathbb{O}$---rather than to $\mathbb{F}$ itself; see  \cite{van2024efficient}.
A natural response to this issue is to enlarge the ambiguity set so that it contains the true underlying distribution. For example, \cite{farokhi2022distributionally} proposes increasing the radius of the Wasserstein ball centered at the noisy empirical distribution $\widehat{\mathbb{F}}^\star$ to cope with locally differentially private data. 
However, such an approach can be overly conservative, as the enlarged ambiguity set may contain distributions that deviate significantly from the latent distribution $\mathbb{F}$.
An alternative strategy is to construct the ambiguity set in the noisy observation space and then \emph{pull it back} to the latent distribution as the inverse image via the noise kernel $\mathbb{O}$, as in~\cite{farokhi2023distributionally}, which uses the total variation distance to construct the ambiguity set.
This paper, instead of using the total variation distance, adopts the same inverse-image principle with Wasserstein balls, which is more suitable for distributions with non-overlapping supports; see \cite{villani2008optimal, peyre2019computational}.

By formulating a noisy-data DRO problem with the ambiguity set using the Wasserstein metric, centered at the noisy empirical distribution~$\widehat{\mathbb{F}}^\star$, we significantly extend the results in~\cite{mohajerin2018data}. 
In fact, their approach can be recovered as a special case of our formulation with noise-free data.
We further demonstrate that, surprisingly, our noisy-data DRO model can leverage the presence of noise to its advantage. Noise effectively enlarges the feasible set, which can lead to higher optimal values. Moreover, we introduce a new measure, the \textit{price of ambiguity} (POA), to quantify the cost of increasing ambiguity. Interestingly, we find that noisy-data DRO can achieve a lower POA, meaning that noise can actually aid decision-making rather than hinder it in the DRO framework.

\subsection{Contributions of the Paper}
We propose a novel noisy-data DRO framework that builds a Wasserstein-ball ambiguity set around~$\widehat{\mathbb{F}}^\star$ in the noisy observation space and then pulls back to the latent (noise-free) space via the inverse image of the noise operator $T_\mathbb{O}$ induced by a known noise kernel. Our main contributions are as follows:
We demonstrate that $(i)$ noisy data may be a ``blessing in disguise": Under mild conditions (e.g., mean-preserving noise), the noisy-data DRO admits a larger decision-feasible region than a direct DRO that treats noisy samples as clean, leading to a higher robust optimal value and incurring a smaller price of ambiguity;  
$(ii)$ we provide a complete theoretical validation of our framework, including a tractable convex reformulation, finite-sample performance guarantees, and asymptotic consistency, and $(iii)$ in the context of fair resource allocation, increasing ambiguity moves solutions toward more equitable allocations while preserving higher aggregate utility than enlarging the fairness parameter directly.

\textit{Paper Organization.}
Section~\ref{sec: Preliminaries} formalizes the DRO problem with noisy data and defines the inverse-image Wasserstein ambiguity set. Section~\ref{section: Statistical Guarantees for Noisy-Data DRO} establishes its statistical performance guarantees, such as finite-sample guarantees and asymptotic consistency. Subsequently,
Section~\ref{section: main results} presents the main results, including a tractable convex reformulation of the noisy-data DRO and offers robust optimality comparisons with its direct DRO, and the introduction of the price of ambiguity along with its dominance properties.
Section~\ref{section: Illustrative Examples} validates the theoretical results through a fair resource allocation example in communication networks.

\subsection{Notations}

Henceforth, we use $\mathbb{R}$ to denote the set of real numbers and $\mathbb{R}_{+}$ to denote the set of nonnegative real numbers.
The set of nonnegative integers is denoted as $\mathbb{N}$.
Random objects are defined on a probability space $(\Omega, \mathcal{F}, \mathbb{P})$, where $\Omega$ is the sample space,~$\mathcal{F}$ is the $\sigma$-algebra of events, and $\mathbb{P}$ is the probability measure. 
The standard inner product of two vectors $x, y \in \mathbb{R}^n$ is denoted as~$\langle x, y \rangle := x^\top y$.
We denote the $\ell_p$-norm of $x$ as~$\|x\|_p :=  (\sum_{i=1}^n |x_i|^p)^{1/p}$, and its associated dual norm as $\|x\|_{p,*} := \sup_{\|z\|_1 \leq 1} \langle x, z \rangle$. Specifically, we write $\|\cdot\|_2$ explicitly for the Euclidean norm; $\|\cdot\|$ denotes $\ell_1$ throughout
The product of two probability distributions~$\mathbb{F}_1$ defined on~$\mathfrak{X}_1$ and~$\mathbb{F}_2$ defined on $\mathfrak{X}_2$, is represented as the distribution $\mathbb{F}_1 \otimes \mathbb{F}_2$ on $\mathfrak{X}_1 \times \mathfrak{X}_2$.

\section{Preliminaries}\label{sec: Preliminaries}
This section provides the necessary preliminaries and then formulates the new noisy-data DRO problem.

\subsection{Noisy Data} \label{subsection: noisy data}
Consider a random vector $X \in \mathbb{R}^n$ with an unknown distribution $\mathbb{F}$ supported on a known compact convex set
$
\mathfrak{X} := \{x \in \mathbb{R}^n : x_{\min} \leq x \leq x_{\max} \},
$
where the bounds $x_{\min}, \, x_{\max} \in \mathbb{R}^n$ are defined componentwise as
$x_{\min} := [x_{\min, 1}, \cdots, x_{\min, n}]^\top$
and
$x_{\max} := [x_{\max, 1}, \cdots,  x_{\max,n}]^\top$.
Let $\widehat{X}_1, \ldots , \widehat{X}_N \in \mathbb{R}^n$ denote~$N$ i.i.d. random samples drawn from the latent (noise-free) distribution $\mathbb{F}$.
Define the noise-free empirical distribution
$
\widehat{\mathbb{F}} := \frac{1}{N} \sum_{j=1}^N \delta_{\widehat{X}_j},
$ 
where~$\delta_{\widehat{X}_j}$ is the Dirac-delta~measure at $\widehat{X}_j$.

However, these i.i.d. noise-free samples $\widehat{X}_j \sim \mathbb{F}$ are not observed directly in practice. 
Instead, due to data uncertainty or measurement errors, one typically only has access to their noisy counterparts~$\widehat{X}^{\star}_j \in \mathbb{R}^n$.

To formalize this, assume the decision maker observes $N$ i.i.d. \emph{noisy} samples. Specifically, for~$j = 1,\dots, N$,  let $\widehat{X}_j^{\star}$  denote the noisy observations of $\widehat{X}_j$, obtained through the conditional distribution~$\mathbb{O}(\cdot \mid \widehat{X}_j)$. The noisy samples are supported on a convex compact set~$\mathfrak{X}^{\star} \subseteq \mathbb{R}^n$, which satisfies $ \mathfrak{X} \subseteq \mathfrak{X}^\star$. 
That is, for~$j = 1,\dots, N$, we have
\begin{align} \label{eq: iid noisy random samples and its distributions}
\begin{cases}
	\widehat{X}_j \sim \mathbb{F}; \\
	\widehat{X}^{\star}_j \sim \mathbb{O}(\cdot \mid \widehat{X}_j)
\end{cases}
\end{align}
The marginal distribution of the noisy random vector $X^{\star}$, call it $\mathbb{F}^{\star}$, can be computed as follows: 
For any measurable set $A \subseteq \mathfrak{X}^\star$,  applying the law of total probability yields
\begin{align} \label{eq: marginal distribution of noisy return}
\mathbb{F}^{\star}(A) 
& := \mathbb{P} (X^\star \in A) \notag\\
& = \int_{x \in \mathfrak{X}}\mathbb{P}( X^{\star} \in A \mid X = x) d\mathbb{F}(x) \notag\\
& = \int_{x \in \mathfrak{X}} \mathbb{O}(A \mid x) d\mathbb{F}(x), 
\end{align}
where $\mathbb{O}(A \mid x):= \mathbb{P}( X^{\star} \in A \mid X = x)$
is the conditional distribution of the event $\{X^\star \in A\}$ with $A \subset \mathfrak{X}^\star$, given the true $X=x$.  
Assume the noise model is \emph{additive}, i.e., $X^\star:= X + E$, where the noise term $E$ is supported on a set $\mathfrak{E}$ and is independent of the latent variable $X$. 
Let $\mathbb{F}_E$ denote the distribution of $E$. In this case, the conditional probability $\mathbb{O}(A \mid x)$ is given by
$$
\mathbb{O}(A \mid x) = \mathbb{P}(X^\star \in A \mid X= x) = \mathbb{P}(x + E \in A \mid X= x)  = \mathbb{P}(E \in A - x) := \mathbb{F}_E(A - x),
$$
where $A - x := \{ z \in \mathbb{R}^n : z + x \in A\}$.
Substituting this into Equation~\eqref{eq: marginal distribution of noisy return}, the marginal distribution of $X^\star$ is the convolution of the distributions of $X$ and~$E$:
\begin{align} 
\mathbb{F}^{\star}(A) 
& = \int_{x\in \mathfrak{X}} \mathbb{F}_E( A- x) \, d\mathbb{F}(x) \notag\\
& = (\mathbb{F}_{E} \star \mathbb{F} )(A).
\end{align}
Additionally, the noisy empirical distribution, denoted as~$\widehat{\mathbb{F}}^\star$, is defined based on $N$ random samples~$\widehat{X}_1^\star, \ldots, \widehat{X}_N^\star$, independently drawn from the noisy distribution~$\mathbb{F}^{\star}$. That is,
$
\widehat{\mathbb{F}}^{\star} := \frac{1}{N} \sum_{j=1}^N \delta_{\widehat{X}^{\star}_j}.
$

\medskip
\begin{remark}[Noiseless Case]\rm 
In the noiseless case, the decision maker observes the noise-free data:~$\widehat{X}_j^\star = \widehat{X}_j$ for all $j$. Hence, the conditional distribution $\mathbb{O}(\cdot \mid x)$ becomes a Dirac measure centered at~$x$, i.e., $\mathbb{O}(A \mid x) = \delta_x(A)$.
Equation~\eqref{eq: marginal distribution of noisy return} then reduces to $\mathbb{F}^{\star}(A) 
= \int_{x \in \mathfrak{X}} \delta_x(A) d\mathbb{F}(x) = \mathbb{F}(A)$, which shows that the observed marginal becomes the true distribution, i.e., $\mathbb{F}^\star = \mathbb{F}$.
\end{remark}

\subsection{Wasserstein Ambiguity Set with Noise}
We now introduce the Wasserstein ambiguity set, which models uncertainty about the true (noiseless) distribution $\mathbb{F}$, based on a finite number of noisy observations.  We begin by recalling the definition of the \emph{Wasserstein metric}, as defined in~\cite{kuhn2025distributionally}, and use it to construct the Wasserstein ball, which serves as an ambiguity set.

\begin{definition}[Wasserstein Metric]\rm
Let $\mathbb{F}, \mathbb{F}' \in \mathcal{M}(\mathbb{R}^n)$ be any two distributions supported on subsets $\mathfrak{X}, \mathfrak{X}' \subseteq \mathbb{R}^n$, respectively, where $\mathcal{M}(\cdot)$ denotes the set of probability measures supported on the indicated space.
The \emph{Wasserstein metric of order $1$}, denoted by $d(\mathbb{F}, \mathbb{F}') : \mathcal{M}(\mathbb{R}^n) \times \mathcal{M}(\mathbb{R}^n) \rightarrow \mathbb{R}_+$, quantifies the discrepancy between two distributions $\mathbb{F}$ and $\mathbb{F}'$. With respect to the $\ell_1$-norm~$\|\cdot\|$, it is defined as:
\begin{align*}
	d(\mathbb{F}, \mathbb{F}') &:= \min_{\Pi \in \mathcal{C}(\mathbb{F}, \mathbb{F}')}  \mathbb{E}^{(\xi, \xi') \sim \Pi} \left[ \| \xi - \xi' \| \right]
\end{align*}
where $\mathcal{C}(\mathbb{F}, \mathbb{F}')$ is the set of all joint distributions $\Pi$ on $\mathbb{R}^n \times \mathbb{R}^n$  with marginals $\mathbb{F}$ and $\mathbb{F}'$.
\end{definition}

Given a reference distribution $\mathbb{F}' \in \mathcal{M}(\mathfrak{X}')$ and radius $\varepsilon > 0$, the associated Wasserstein ball is defined as
$
\mathcal{B}_{\varepsilon} ({\mathbb{F}}') := \left\{ \mathbb{F} \in \mathcal{M}(\mathfrak{X}): d(\mathbb{F}, {\mathbb{F}}') \leq \varepsilon \right\}.
$
Thus, the Wasserstein ball centered at the \textit{noisy empirical distribution $\widehat{\mathbb{F}}^\star$} defined in Section~\ref{subsection: noisy data} is given by
$$
\mathcal{B}^\star_{\varepsilon} (\widehat{\mathbb{F}}^{\star})  := \left\{ \mathbb{F}^{\star} \in \mathcal{M}(\mathfrak{X}^\star): d(\mathbb{F}^{\star}, \widehat{\mathbb{F}}^{\star}) \leq \varepsilon \right\}.
$$ 
To recover ambiguity over the latent (noise-free) distribution $\mathbb{F}$, we define the noise operator $T_\mathbb{O}: \mathcal{M}(\mathfrak{X}) \to \mathcal{M}(\mathfrak{X}^\star)$ with $(T_\mathbb{O}(\mathbb{F}))(A)  := \int_{x \in \mathfrak{X}} \mathbb{O}(A \mid x) d \mathbb{F}(x)$ for any measurable set $A$, resulting in $\mathbb{F}^\star = T_\mathbb{O}(\mathbb{F})$. Then, the corresponding ambiguity~set over $\mathbb{F}$ induced by the Wasserstein ball on $\mathbb{F}^\star$ can be defined as the inverse image of the Wasserstein ball centered at the noisy empirical distribution~$\widehat{\mathbb{F}}^\star$: 
\begin{align}\label{eq: noisy Wasserstein ball}
\mathcal{B}_{\varepsilon, \mathbb{O}} (\widehat{\mathbb{F}}^{\star}) 
& := T_\mathbb{O}^{-1}( \mathcal{B}_{\varepsilon}^\star (\widehat{\mathbb{F}}^{\star})) \notag \\
&
= \left\{ \mathbb{F} \in \mathcal{M}(\mathfrak{X}) :  d(T_\mathbb{O}(\mathbb{F}), \widehat{\mathbb{F}}^\star) \leq \varepsilon \right\}
= \left\{ \mathbb{F} \in \mathcal{M}(\mathfrak{X}) :  \mathbb{F}^\star  \in \mathcal{B}_{\varepsilon}^\star (\widehat{\mathbb{F}}^{\star}) \right\}.
\end{align}
This set contains all latent distributions $\mathbb{F}$, whose corresponding noisy distribution, $\mathbb{F}^\star = T_\mathbb{O}(\mathbb{F})$ is within a Wasserstein distance of $\varepsilon$ from the noisy empirical distribution $\widehat{\mathbb{F}}^\star$.

\subsection{DRO Problem with Noisy Data}
We now formulate the distributionally robust optimization (DRO) problem faced by a decision maker who only observes noisy data. 

The decision maker does not observe the latent $X_j$, but rather noisy observations $X^\star_j$, whose empirical distribution $\widehat{\mathbb{F}}^\star$ was introduced in Section~\ref{subsection: noisy data}. 
The latent distribution $\mathbb{F}$ is assumed to lie within the ambiguity set $\mathcal{B}_{ \varepsilon, \mathbb{O}}(\widehat{\mathbb{F}}^\star)$ defined in~\eqref{eq: noisy Wasserstein ball}.
The decision maker adopts a \textit{distributionally robust optimization} (DRO) approach and selects a decision $w \in \mathcal{W}$ to maximize worst-case expected utility under all latent distributions $\mathbb{F} \in \mathcal{B}_{ \varepsilon, \mathbb{O}}(\widehat{\mathbb{F}}^\star)$, where $\mathcal{W} \subseteq \mathbb{R}^n$ denotes the compact feasible set of decisions.
This leads to the following noisy-data DRO problem.

\medskip
\begin{problem}[Noisy-Data DRO Problem] \rm
Given a continuous, concave, and nondecreasing utility function $U: \mathbb{R} \to \mathbb{R}$, a Wasserstein radius~$\varepsilon >0 $, and a noisy empirical distribution $\widehat{\mathbb{F}}^\star,$ we consider the following distributionally robust optimization problem under noisy information:
\begin{align}\label{eq: noise DRO}
	g_{\rm noise}^*(\varepsilon) := 
	\sup_{w \in \mathcal{W}} \, 
	\inf_{\mathbb{F} \in \mathcal{B}_{ \varepsilon, \mathbb{O}}(\widehat{\mathbb{F}}^{\star})} \mathbb{E}^{\mathbb{F}}  \left[  U \left( \langle w, X \rangle \right) \right],
\end{align} 
where the expectation operator is taken with respect to the latent (noiseless) distribution~$\mathbb{F}$. 
\end{problem}

\begin{remark} \rm
The objective $g_{\text{noise}}^*(\cdot)$ captures the worst-case expected utility over the Wasserstein ambiguity set of distributions induced by noisy data. This framework reflects a risk-aware decision-making approach under partial observability.  Henceforth, the optimal solution, if it exists, is denoted by $w_{\text{noise}}^* \in \mathcal{W}$.
\end{remark}

\section{Statistical Guarantees for Noisy-Data DRO} \label{section: Statistical Guarantees for Noisy-Data DRO}
This section provides the theoretical validation for the noisy-data DRO model formulated in Section~\ref{sec: Preliminaries}. We establish new results that the model is statistically well-behaved by deriving finite-sample performance guarantees and then proving its asymptotic consistency.

\subsection{Finite-Sample Guarantees}
Let $\widehat{w}_{\text{noise}, N}^{*}$ be the optimal solution of the noisy-data DRO problem~\eqref{eq: noise DRO}, obtained from an in-sample dataset $\widehat{\mathfrak{X}}^\star_N$ consisting of i.i.d. noisy samples $\widehat{X}_1^{\star}, \dots, \widehat{X}_N^{\star}$.
The out-of-sample performance of~$\widehat{w}_{\text{noise}, N}^{*}$ is defined as $\mathbb{E}^\mathbb{F} [U(\langle \widehat{w}_{\text{noise}, N}^{*}, X \rangle)]$, which is the expected utility of $\widehat{w}_{\text{noise}, N}^{*}$ when evaluated on a new sample drawn from $\mathbb{F}$.
Consistent with the DRO literature,~\cite{mohajerin2018data, gao2023finite}, we establish an out-of-sample performance guarantee below to demonstrate that our ``pull-back" model is well-behaved. Specifically, let $\widehat{g}_{\text{noise}, N}^{*}$ denote the in-sample optimal values of~\eqref{eq: noise DRO}, we~seek:
\begin{align}\label{eq: finite sample guarantee}
\mathbb{P}^{\star N} \{ \widehat{\mathfrak{X}}^{\star}_N : \widehat{g}^{* }_{\text{noise}, N} \leq \mathbb{E}^\mathbb{F} [U(\langle \widehat{w}_{\text{noise}, N}^{*}, X \rangle)] \} \geq 1 - \beta,
\end{align}
where $\mathbb{P}^{\star N}$ denotes the $N$-fold product distribution of a distribution $\mathbb{F}^{\star}$ on $\mathfrak{X}^\star$, the in-sample dataset~$ \widehat{\mathfrak{X}}^{\star}_N = \{\widehat{X}_1^{\star}, \dots, \widehat{X}_N^{\star} \}$ is a random variable governed by $\mathbb{P}^{\star N} \in \mathcal{M}(\mathfrak{X}^{\star N})$, where the support~$\mathfrak{X}^{\star N}$ is the $N$-fold Cartesian product of set~$\mathfrak{X}^{\star}$, and $\beta \in (0,1)$ is a confidence parameter.

\begin{theorem}[Measure Concentration]\label{theorem: Measure concentration}
For any choice of $a>1$, and all $N \geq 1$, $n \neq 2$, and~$\varepsilon >0$, the following two statements hold.

$(i)$  We have
\begin{align}\label{ineq: measure concentration}
	\mathbb{P}^{\star N} \{ \widehat{\mathfrak{X}}^{\star}_N : d(\mathbb{F}^{\star}, \widehat{\mathbb{F}}^{\star}) \geq \varepsilon \} \leq
	\begin{cases}
		c_1 \exp (- c_2 N \varepsilon^{\max\{n,2\}}) &\text{ if } \varepsilon \leq 1;\\
		c_1 \exp (-c_2 N \varepsilon^a) &\text{ if } \varepsilon > 1,
	\end{cases}
\end{align}
where $c_1, c_2$ are positive constants that only depend on $a$,~$n$, and the finite moment $A:=\mathbb{E}^{\mathbb{F}^\star}[\exp \|X^\star\|^a].$

$(ii)$ Let $\varepsilon_N (\beta)$ denote the smallest radius of the Wasserstein ball that $\mathbb{F}^\star$ lies inside the ball with confidence at least $1 - \beta$. Then
\begin{align*}
	\varepsilon_N (\beta) =
	\begin{cases}
		\left( \frac{\log(c_1 \beta^{-1})}{c_2 N}\right)^{ \tfrac{1}{\max \{n,2\}} } &\text{ if } N \geq \frac{\log(c_1 \beta^{-1})}{c_2}; \\
		\left( \frac{\log(c_1 \beta^{-1})}{c_2 N}\right)^{ \tfrac{1}{ a} } & \text{ if } N < \frac{\log(c_1 \beta^{-1})}{c_2} .
	\end{cases}
\end{align*}
\end{theorem}

\begin{proof}
See Appendix~\ref{appendix: Statistical Guarantees for Noisy-Data DRO}.
\end{proof}

\begin{remark} \rm
In general, Theorem~\ref{theorem: Measure concentration} requires the so-called \emph{light-tailed} condition; see \cite[Assumption 3.3]{mohajerin2018data}. However, we note that, because the support of the noisy data, $\mathfrak{X}^\star$, is assumed to be compact, the light-tailed condition required is satisfied~automatically.
\end{remark}

Theorem~\ref{theorem: Measure concentration} and $\varepsilon_N (\beta)$ lead to Corollary~\ref{corollary: Concentration Inequality}.

\begin{corollary}[Concentration Inequality]\label{corollary: Concentration Inequality}
Let  $\varepsilon_N (\beta)$  be the smallest radius of the Wasserstein ball that contains $\mathbb{F}^\star$ with confidence at least $1 - \beta$ as defined in Theorem~\ref{theorem: Measure concentration}.  Then we have
$\mathbb{P}^{\star N} \{ \widehat{\mathfrak{X}}^{\star}_N : \mathbb{F} \in \mathcal{B}_{ \varepsilon_N (\beta), \mathbb{O}}(\widehat{\mathbb{F}}^{\star})\} \geq  1-\beta$.
\end{corollary}

\begin{proof}
See Appendix~\ref{appendix: Statistical Guarantees for Noisy-Data DRO}.
\end{proof}

\begin{theorem}[Finite Sample Guarantee under Noisy Data]\label{theorem: Finite Sample Guarantee under Noisy Data}
Let $\widehat{g}^*_{\text{noise}, N}$ and $\widehat{w}^*_{\text{noise}, N}$ be the optimal value and optimal solution of the noisy-data DRO problem~\eqref{eq: noise DRO}.
Let $\beta \in (0,1)$. Then the finite sample guarantee~\eqref{eq: finite sample guarantee} holds. That is,
\begin{align*}
	\mathbb{P}^{\star N} \{ \widehat{\mathfrak{X}}^{\star}_N : \widehat{g}^{* }_{\text{noise}, N} \leq \mathbb{E}^\mathbb{F} [U(\langle \widehat{w}_{\text{noise}, N}^{*}, X \rangle)] \} \geq 1 - \beta.
\end{align*}
\end{theorem}

\begin{proof}
See Appendix~\ref{appendix: Statistical Guarantees for Noisy-Data DRO}.
\end{proof}

\subsection{Asymptotic Consistency}

Motivated by the relationship between noisy and noiseless distribution distances established in \cite[Lemma A.4]{van2024efficient}, the following lemma extends the result to the Wasserstein setting.

\begin{lemma}[Noise to Latent Distance]\label{lemma: Noise to Latent Distance}
Let the noisy counterpart be generated by an additive noise model $X^\star = X + E$, where the noise $E$ has distribution $\mathbb{F}_E$ and is independent of $X$. 
Let $\mathbb{Q},\mathbb{P} \in \mathcal{M}(\mathfrak{X})$ be any two distributions supported on a compact set $\mathfrak{X}$, and $\mathbb{Q}^\star,\mathbb{P}^\star \in \mathcal{M}(\mathfrak{X}^\star)$ be their noisy counterparts, supported on~$\mathfrak{X}^\star$. 
For~$t \in \mathbb{R}^n$, let $\varphi_{\mathbb{F}_E} (t)$ be the characteristic function of the distribution $\mathbb{F}_E$, and assume that the set $\{t \in \mathbb{R}^n: \varphi_{\mathbb{F}_E}(t) = 0\}$ has Lebesgue measure zero. 
For any sequence $\{\mathbb{Q}_N\}_{N=1}^\infty \subset \mathcal{M}(\mathfrak{X})$ and any $\mathbb{P} \in \mathcal{M}(\mathfrak{X})$, we have
$$
\lim_{N \to \infty} d(\mathbb{Q}_N^\star, \mathbb{P}^\star) =0
\; \text{ implies } \;
\lim_{N \to \infty}  d(\mathbb{Q}_N, \mathbb{P}) =0.
$$
\end{lemma}

\begin{proof}
See Appendix~\ref{appendix: Statistical Guarantees for Noisy-Data DRO}.
\end{proof}

We now establish the novel asymptotic consistency in the presence of noise.

\begin{theorem}[Asymptotic Consistency]\label{theorem: Asymptotic Consistency} 
Let $\beta_N \in (0, 1)$ for $N \in \mathbb{N}$ which satisfies $\sum_{N=1}^{\infty} \beta_N < \infty$ and let $\varepsilon_N(\beta_N)$ be chosen such that $\lim_{N \to \infty} \varepsilon_N (\beta_N) = 0$. Let $\widehat{g}^*_{\text{noise}, N}$ and $\widehat{w}_{\text{noise}, N}^*$ denote the in-sample optimal value and an optimal solution of the noisy-data DRO problem~\eqref{eq: noise DRO} with $N$ samples, 
and let $g^*_{SO} := \sup_{w\in \mathcal{W}} \mathbb{E}^\mathbb{F} [U( \langle w, X \rangle)]$ be the optimal value of the stochastic optimization problem under the latent (noiseless) distribution $\mathbb{F}$.

$(i)$
If $U ( \langle w, x\rangle)$ is $L$-Lipschitz continuous on $\mathfrak{X}$,
the sequence $\widehat{g}_{\text{noise}, N}^*$ converge towards $ g^*_{SO}$ with probability one, i.e., 
$$
\mathbb{P}^{\star, \infty} \{ \widehat{\mathfrak{X}}^{\star}_N:\lim_{N \to \infty} \widehat{g}^*_{\text{noise}, N} = g^*_{SO} \} = 1.
$$

$(ii)$ In addition, for the compact feasible set $\mathcal{W}$,
any accumulation point of sequence~$\{\widehat{w}^*_{\text{noise}, N}: N \in \mathbb{N}\}$ is an optimal solution for $g^*_{SO}$ probability one.
\end{theorem}

\begin{proof}
See Appendix~\ref{appendix: Statistical Guarantees for Noisy-Data DRO}.
\end{proof}

\section{Theoretical Reformulation and Dominance Analysis} \label{section: main results}
This section provides a tractable convex reformulation of the noisy-data DRO problem. We show how noise, typically viewed as detrimental, can expand the feasible set and improve the optimal value. We provide theoretical tractability and robust optimality comparisons for noisy-data DRO and direct DRO models. In this section, we consider the additive noise model $X^\star = X + E$ with $\mathbb{E}[\|E\|]<\infty$.\footnote{ Under this additive noise model assumption, the map $x \mapsto \int \|x^\star - \widehat{x}_j^\star\| d\mathbb{O}(x^\star \mid x) = \mathbb{E}\|x + E - \widehat{x}_j^\star\|$, which is 1-Lipschitz, hence continuous in $x$. More generally, it is possible to consider the noise kernel $\mathbb{O}$ to be \emph{Feller} kernel, i.e., $x \mapsto \int f(x^\star) d\mathbb{O}(x^\star \mid x)$ is continuous for every bounded continuous $f$.}

\subsection{Reformulation of the Noisy-Data DRO Problem}\label{sec: Convex Approximation of the Noisy-Data DRO Problem}

We now show that the noisy-data DRO problem can be reformulated as a tractable program, thereby generalizing the existing result in~\cite{mohajerin2018data}.

\begin{lemma}[A Convex Reformulation of Noisy-Data DRO]\label{lemma: An Equivalent Representation of Noisy-Data DRO} 
For any $\varepsilon > 0$, the noisy-data DRO problem~\eqref{eq: noise DRO} can be reformulated as the following convex program: 
\begin{align}\label{eq: convex DRO noise}
	&\displaystyle \sup_{w \in \mathcal{W},\lambda \geq 0, a_j, \forall j} - \lambda \varepsilon +  \frac{1}{N} \sum_{j=1}^{N} a_j\\
	&{\rm s.t.} \;  \displaystyle \inf_{x \in \mathfrak{X}} \left( U \left( \langle w, x \rangle \right) + \lambda \int_{x^\star \in \mathfrak{X}^\star}  \|x^{\star}-\widehat{x}_j^\star\| d\mathbb{O}(x^{\star} \mid x) \right)  \geq a_j,\quad j = 1, \dots, N, \notag
\end{align}
where $\widehat{x}_j^\star$ is the $j$-th observed noisy sample vector.
\end{lemma}

\begin{proof}
See Appendix~\ref{appendix: proof in Section: main results}.
\end{proof}

\begin{remark}[Baseline Direct DRO Model]\label{remark: noise free DRO} \rm
As a baseline for comparison, we consider a \textit{Direct DRO} model, which naively applies a standard data-driven DRO formulation to the observed samples while treating the observed noisy samples~$\widehat{x}_j^\star$ as if they were noise-free.
The direct DRO model is equivalent to our Noisy-Data DRO formulation in the special case where the noise is a Dirac measure. For any measurable set $A \subseteq \mathfrak{X}^\star$, we have~$\mathbb{O}(A \mid x) = \delta_x(A)$. In this hypothetical case, where~$\widehat{x}_j^\star = \widehat{x}_j$, the program from Lemma \ref{lemma: An Equivalent Representation of Noisy-Data DRO} simplifies to:
\begin{align} \label{eq: DRO no noise}
	&\displaystyle\sup_{w \in \mathcal{W}, \lambda \geq 0, a_j, \forall j} - \lambda \varepsilon + \frac{1}{N} \sum_{j=1}^{N} a_j \\
	&{\rm s.t.} \;  \inf_{x \in \mathfrak{X}} \left(
	U( \langle w, x \rangle )  +  \lambda \| x-\widehat{x}_j^\star \| \right) \geq a_j, \quad j = 1, \ldots, N. \notag
\end{align}
To see this, let $x\in \mathfrak{X}$. Given that $ \mathbb{O}(A \mid x) = \delta_x(A)$, it follows that the integral in Lemma~\ref{lemma: An Equivalent Representation of Noisy-Data DRO} collapses to a single norm $\| x - \widehat{x}_j^\star\|$. 
Equation~\eqref{eq: DRO no noise} is consistent with \cite{mohajerin2018data}, where the model does not account for noise. Both ``noisy" and ``noiseless" are treated as if they were noiseless.
\end{remark}

\subsection{Sensitivity of the Optimal Value to the Robustness Radius}
We now analyze the sensitivity of the optimal value of the noisy-data DRO problem with respect to the Wasserstein radius $\varepsilon$. This analysis provides an economic interpretation of the Lagrange multiplier $\lambda$ associated with the size of the ambiguity set.

\begin{definition}[Value Function] \rm
Let the value function $g_{\rm noise}^*(\varepsilon): \mathbb{R}_+ \to \mathbb{R}$ denote the optimal value of the noisy-data DRO problem \eqref{eq: noise DRO} as a function of the Wasserstein radius $\varepsilon$. Let~$(w^*(\varepsilon), \lambda^*(\varepsilon), \{a_j^*(\varepsilon)\})$ denote a corresponding optimal solution to the equivalent problem \eqref{eq: convex DRO noise}.
\end{definition}

The following theorem formally characterizes the relationship between the value function and the optimal Lagrange multiplier.

\begin{theorem}[Sensitivity of the Value Function]\label{theorem: Sensitivity of the Value Function}
Assume that for a given $\varepsilon > 0$, the optimal solution to problem \eqref{eq: convex DRO noise} is unique and that $\lambda^*(\varepsilon)$ is continuous at $\varepsilon$. Then the value function~$g_{\rm noise}^*(\varepsilon)$ is differentiable at $\varepsilon$ and its derivative is given by:
$$
\frac{d g_{\rm noise}^*(\varepsilon)}{d\varepsilon} = -\lambda^*(\varepsilon).
$$
\end{theorem}

\begin{proof}
See Appendix~\ref{appendix: proof in Section: main results}.
\end{proof}

\begin{remark} \rm
Theorem \ref{theorem: Sensitivity of the Value Function} established that the optimal Lagrange multiplier, $\lambda^*(\varepsilon)$,  quantifies the marginal decrease in the worst-case utility corresponding to an incremental increase in the ambiguity radius $\varepsilon$.  Therefore, $\lambda^*(\varepsilon)$ can be interpreted as the ``shadow price of robustness."	That is, for ensuring robustness per unit, in terms of $\varepsilon$, we must ``pay" a price of $\lambda^*$ in the form of the reduced guaranteed utility.
\end{remark}

\subsection{Comparison of Optimal Values across DRO Models}
To investigate how incorporating noisy distributions affects the resulting optimization outcomes, we compare the optimal values of two models: $(i)$ the Direct DRO problem and $(ii)$ the Noisy-Data DRO problem. 
For brevity in referring to the models, and for $\varepsilon >0$, Table~\ref{table: model notation} summarizes the shorthand notations for problems and the associated optimal value notations in each case.

\begin{table}[htbp]
\scriptsize
\centering
\caption{Shorthand Notations for the Two DRO Models}
\label{table: model notation}
\begin{tabular}{l l l l}
	\toprule
	DRO Model  & Shorthand & Optimal Value \\ 
	\midrule
	Direct DRO problem (Remark~\ref{remark: noise free DRO}) & $\texttt{DRO}$ & $g^*(\varepsilon)$\\ 
	Noisy-Data DRO problem (Lemma~\ref{lemma: An Equivalent Representation of Noisy-Data DRO}) & $\texttt{DRO}_{\rm {noise}}$ & $g_{\rm noise}^*(\varepsilon)$ \\ 
	\bottomrule
\end{tabular}
\end{table}

For the subsequent dominance result, we assume the noise is mean-preserving. That is, the conditional expectation of a noisy observation given the true value is the true value itself, i.e.,~$\mathbb{E}^{\mathbb{O}(\cdot \mid x)} [X^\star \mid X=x] = x$ for all~$x \in \mathfrak{X}$.

\begin{theorem}[Optimal Values Dominance]\label{theorem: opt values comparison}
For $\varepsilon>0$, the optimal values of the two DRO models described in Table~\ref{table: model notation} satisfy 
$g^*_{\rm noise}(\varepsilon) \geq g^*(\varepsilon)$.
\end{theorem}

\begin{proof}
See Appendix~\ref{appendix: proof in Section: main results}.
\end{proof}

\begin{remark} \rm
For the theorem, it is important to note that the optimal Lagrange multipliers $\lambda^{*}$ (direct) and $\lambda^{*}_{\mathrm{noise}}$ (noisy) \emph{need not coincide}.   
The theorem relies only on pointwise feasibility dominance for every fixed $(w, \lambda)$, not on equality of optimal dual solutions.
\end{remark}

The dominance result of Theorem~\ref{theorem: opt values comparison} rests on the assumption of mean-preserving noise. We now quantify how this result is affected if the noise exhibits a bounded bias.

\begin{proposition}[Robustness Bound under Biased Noise] \label{proposition: Robustness Bound under Biased Noise}
For $\varepsilon >0$. Let the noise model be subject to a bias $b(x)$ such that $\mathbb{E}^{\mathbb{O}(\cdot \mid x)}[X^\star] = x + b(x)$ with $\|b(x)\| \leq \delta$ for some $\delta \ge 0$ and all $x \in \mathfrak{X}$. Let $g_{\text{noise}, \delta}^* (\varepsilon)$ denote the optimal value of the noisy-data DRO problem under this biased noise. 
Let $g^*(\varepsilon)$ be the optimal value of the corresponding direct DRO problem. Then the following bound holds:
$$
g_{\text{noise}, \delta}^* (\varepsilon) \geq g^*(\varepsilon) - \lambda^*\delta,
$$
where $\lambda^*$ is the optimal Lagrange multiplier from the direct DRO problem.
\end{proposition}

\begin{proof}
See Appendix~\ref{appendix: proof in Section: main results}.
\end{proof}

\begin{remark} \rm
Proposition~\ref{proposition: Robustness Bound under Biased Noise} quantifies the impact of bias. The dominance of the noisy-data DRO model over the direct one, $g_{\text{noise}, \delta}^* (\varepsilon) \ge g^*(\varepsilon)$, is guaranteed only if the ``benefit" from the noise's variance outweighs the penalty from its bias, $\lambda^*\delta$. If the bias $\delta$ is sufficiently large, the noisy-data DRO model can indeed become more conservative than its direct counterpart.
\end{remark}

\subsection{Price of Ambiguity}\label{section: Price of Ambiguity}
Motivated by~\cite{bertsimas2004price,bertsimas2011price}, we introduce a new metric, the \textit{price of ambiguity}, to quantify the loss in optimal value due to distributional ambiguity.
In particular, we define $\text{SYSTEM}$ as the optimal value of the sample average approximation (SAA) under the empirical distribution $\widehat{\mathbb{F}}^\star$. This corresponds to the direct DRO model with $\varepsilon=0$, where the observed data is used without accounting for ambiguity:
$$
\text{SYSTEM} := \max_{w \in \mathcal{W}} \mathbb{E}^{\widehat{\mathbb{F}}^\star} [U(\langle w, X^\star\rangle)] = \max_{w \in \mathcal{W}} \frac{1}{N} \sum_{j=1}^N U(\langle w, \widehat{x}_j^\star \rangle).
$$
For $\varepsilon>0$, let $g^*(\varepsilon)$ and $g_{\rm noise}^*(\varepsilon)$ be the optimal values of the direct and noisy-data DRO problems, respectively.
The corresponding price of ambiguity, denoted by $\mathtt{POA}$, measures the relative reduction in the optimal value caused by introducing ambiguity of radius $\varepsilon >0$:
\begin{align*}
\mathtt{POA}(\varepsilon) = \frac{\text{SYSTEM} - g^*(\varepsilon)}{\text{SYSTEM}}
\quad \text{ and } \quad 
\mathtt{POA}_{\rm noise}(\varepsilon)  =  \frac{\text{SYSTEM} - g_{\rm noise}^*(\varepsilon)}{\text{SYSTEM}}.
\end{align*}
The direct DRO model reduces to SAA when $\varepsilon = 0$, since its ambiguity set reduces to a singleton containing only $\widehat{\mathbb{F}}^\star$.
However, the noisy-data DRO behaves differently. Specifically, with zero radius, the ambiguity set~\eqref{eq: noisy Wasserstein ball} used in noisy-data DRO is generally not a singleton. 
To see this, fix $\varepsilon  =0$. We observe that the noisy ambiguity set is given by
\begin{align}
\mathcal{B}_{\varepsilon=0, \mathbb{O}}(\widehat{\mathbb{F}}^\star) 
&= \{\mathbb{F} \in \mathcal{M}(\mathfrak{X}): d(\mathbb{F}^\star, \widehat{\mathbb{F}}^\star) \leq 0\} \notag \\
&= \{\mathbb{F} \in \mathcal{M}(\mathfrak{X}): \mathbb{F}^\star = \widehat{\mathbb{F}}^\star \} \notag \\
&=\left\{ \mathbb{F} \in \mathcal{M}(\mathfrak{X}): \int_{x \in \mathfrak{X}} \mathbb{O}(A \mid x) \, d\mathbb{F}(x) = \widehat{\mathbb{F}}^\star(A) \text{ for all measurable } A \subseteq \mathfrak{X}^\star \right\}. \notag 
\end{align}

The following lemma establishes fundamental properties of the price of ambiguity: its monotonicity with respect to the ambiguity radius $\varepsilon$, and the dominance of the direct DRO model's $\mathtt{POA}$ over that of the noisy-data DRO model.

\begin{lemma}[Dominance of the Price of Ambiguity] \label{lemma: Dominance of the Price of Ambiguity}
If SYSTEM $>0$, then the following statements hold:

$(i)$ the price of ambiguity $\mathtt{POA}(\varepsilon)$ and $\mathtt{POA}_{\text{noise}}(\varepsilon)$ are nondecreasing in~$\varepsilon$.

$(ii)$    For $\varepsilon >0$, we have
$   
\mathtt{POA}(\varepsilon) \geq \mathtt{POA}_{\rm noise}(\varepsilon).
$
\end{lemma}

\begin{proof}
See Appendix~\ref{appendix: proof in Section: main results}.
\end{proof}

\begin{remark}[The Price of Ambiguity as a ``Blessing in Disguise"] \rm
Part~$(ii)$ in Lemma~\ref{lemma: Dominance of the Price of Ambiguity} highlights another sense in which noisy data can be a ``blessing in disguise." It shows that the noisy-data DRO model is not only less conservative (yielding a higher optimal value $g_{\rm noise}^\star$ per Theorem~\ref{theorem: opt values comparison}) but is also less costly in terms of its price of ambiguity. In essence, the relative penalty for a given level of robustness $\varepsilon$ is smaller when the noise structure is explicitly modeled. 
\end{remark}

\section{Illustrative Examples} \label{section: Illustrative Examples}
To substantiate our theoretical findings, this section presents an example in the domain of fair resource allocation in communication networks. Fairness in allocation problems has been studied across various fields, including social sciences, welfare economics, communications, and engineering; see~\cite{bertsimas2011price}. 
While our theory is developed in a continuous setting, we discretize the result for numerical implementation.
In this section, we demonstrate that $(i)$ incorporating noisy data within our DRO framework can lead to superior optimal values, as a direct consequence of the theoretically established enlargement of the feasible set; $(ii)$ these superior values, in turn, imply that noisy-data DRO attain a lower price of ambiguity; and $(iii)$ fairness and ambiguity are conceptually linked, as increasing either the fairness or ambiguity parameter ultimately helps achieve more equitable~allocation.

\subsection{Distributionally Robust Fair Resource Allocation Problem}
We consider the problem faced by a central decision maker who aims to allocate a set of scarce resources \emph{fairly} among $n$ self-interested entities.
Let $w$ be the allocation strategy, where $w_i \geq 0$ represents the fraction of resources allocated to entity $i$, and the entire resource is fully allocated, i.e., $\sum_{i=1}^n w_i = 1$. 
Let $X_i$ be a random variable characterizing the effectiveness or state of entity~$i$.
To formalize the trade-off between efficiency and fairness in allocation, we adopt the $\alpha$-fairness framework, as defined in~\cite{altman2008generalized}, as follows.

\begin{definition}[$\alpha$-Fairness Utility Function]\rm \label{definition: modified alpha fairness}
Given $\alpha \geq 0$, an allocation $w$ is said to be \emph{degree of $\alpha$-fair} if it maximizes the aggregate utility 
$$
\mathbf{U}(\alpha, w, x) := \sum_{i=1}^n U_i  (\alpha, w_i, x_i),
$$
where the utility function of entity $i$ is defined as
\begin{align*}
	U_i(\alpha, w_i, x_i) :=
	\begin{cases}
		\frac{1}{1 - \alpha} \left( \left( u_i( w_i,x_i ) \right)^{1 - \alpha} - 1 \right) , & \alpha \in [0, \infty ), \ \alpha \neq 1; \\
		\log (u_i( w_i,x_i )), & \alpha = 1,
	\end{cases}
\end{align*} 
where $u_i (\cdot)>0$ is the utility of entity $i$. This form, see \cite{altman2008generalized}, ensures continuity in~$\alpha$ at $\alpha = 1$.
\end{definition}

\begin{remark}\rm
The parameter $\alpha \geq 0$ controls the balance between fairness and efficiency. Specifically, when~$\alpha = 0$, the objective reduces to maximizing the aggregate utility $u_i(\cdot)$, corresponding to the most \emph{efficient} strategy under the utilitarian criterion. 
When $\alpha = 1$, it corresponds to proportional fairness.  As~$\alpha \rightarrow \infty$, the objective prioritizes fairness and approaches the fairest allocation--the max-min criterion. See~\cite{xinying2023guide} for a comprehensive survey of this~topic.
\end{remark}

\textit{Distributionally Robust Fair Resource Allocation Problem.}
We now extend the $\alpha$-fair resource allocation in wireless communication networks to the DRO setting to account for uncertainty and noise in the distribution:
$$
\max_{w \in \mathcal{W}} \, 
\inf_{\mathbb{F} \in \mathcal{B}_{ \varepsilon}(\widehat{\mathbb{F}}^{\star})} \mathbb{E}^{\mathbb{F}}  \left[  \mathbf{U}(\alpha, w, x) \right],
$$
where the utility depends on the signal-to-noise ratio (SNR) resulting from power assignments.
Let~$w_i \geq 0$ be the proportion of total transmission power allocated to user $i$, with the feasible set~$\mathcal{W}:= \{w \in \mathbb{R}^n: w_i \geq 0, \sum_{i=1}^n w_i = 1\}$.
Let $X_i$ be the ratio of channel gain to noise power for user $i$, capturing the randomness in the communication environment due to channel fading and background noise. The utility of $i$-th entity is defined as $ u_i( w_i, x_i):= x_i w_i + 1$, which represents the shifted SNR under the $\alpha$-fairness framework.
When $\alpha = 0$, the objective function reduces to the SNR maximization.
When $\alpha = 1$, the objective function corresponds to throughput maximization.

\subsection{Numerical Illustrations}

\textit{Dataset.}
We utilize the Social-Aware Resource Allocation (SARA) dataset, as described in \cite{ziya2024sara}, in a Device-to-Device (D2D) communication network. 
The dataset contains various information such as device and user attributes, including device IDs, user types (Regular, VIP, High-Demand), social influence scores, and the identifier of the base station (one of four) associated with each user, as well as network and channel parameters, including channel gain, noise power, and SNR.
For each base station, we use the channel gain divided by noise power\footnote{To form the $\alpha$-fair assignment of SNR, both values of channel and noise power are converted to linear scale: for example, if $x_{\rm dB}$ is in decibels, then $x = 10^{x_{\rm dB} / 10}$.} as observed data points~$\widehat{x}_j^{\star}$, where~$j = 1, \dots, N$ and $N$ ranges from 70 to 90 depending on the base station.
We normalize the observed data using min-max normalization, rescaling it into a feature range from 0 to~1. The rescaled value is defined as~$x_{\rm scaled} = \frac{x - x_{\min}}{x_{\max} - x_{\min}}$.
The ambiguity radius is set to~$\varepsilon \in [0.01, 0.1]$, and the fairness parameter is chosen from~$\alpha \in [0, 20]$ along with the additional value~$\alpha = 100$, which serves as an approximation of max-min fairness.
The noise distribution $\mathbb{O}$ is uniform (U), i.e., $\mathbb{P}(E = e) = \frac{1}{|\mathfrak{E}|}$ for all $e \in \mathfrak{E}$, where $|\mathfrak{E}|$ denotes the cardinality of the set $\mathfrak{E}$. 
The noise set is~$\mathfrak{E} = \{ e \in \mathbb{R}^n : a \leq e_i \leq b, \ i = 1, \ldots, n\} $, where $-1 < a  < b < \infty$. 
In this example, we set~$a = -0.01$ and $b = 0.01$.

The goal is to determine the power allocation proportion $w$ across three user types: Regular, VIP, and High-Demand, using the two models: $\texttt{DRO}_{\rm {noise}}$ and $\texttt{DRO}$ to examine the effects of fairness, ambiguity, and noise.
For the mathematical formulation, see Appendix~\ref{sec: Fairness Model Formulations}.
The $\texttt{DRO}_{\rm{noise}}$ model applies Lemma~\ref{lemma: An Equivalent Representation of Noisy-Data DRO}, assuming that the data is perturbed by uniform noise.
In contrast, the $\texttt{DRO}$ model applies Remark~\ref{remark: noise free DRO}, treating the data as if it were noiseless.
We are supplied with data samples $\widehat{x}^\star_j$. The procedure for generating the data in support is as follows: first, define the support~$\mathfrak{X}$. Since the data are rescaled to lie in $[0,1]$, we discretize $\mathfrak{X}$ into $\{0, 0.25, 0.5, 0.75, 1\}$. Next, we use the kernel $\mathbb{O}$ to generate multiple noise samples $e$, and for each, obtain $x^\star = x + e$.
Our results demonstrate that increasing either fairness or robustness in the model setup leads to optimal solutions that allocate resources among three user types in a manner that resembles an equitable distribution. Moreover, incorporating noise results in a higher aggregate utility.

\subsubsection{Fairness and Ambiguity.}

Figures~\ref{figure: SARA epsilon}--\ref{figure: SARA alpha} show the optimal allocation strategies of base station 1 under the two models:
$\texttt{DRO}$ and $\texttt{DRO}_{\rm {noise}}$,
as we vary the radius of the ambiguity set $\varepsilon$ and the fairness parameter $\alpha$, respectively.
In Figures~\ref{figure: SARA epsilon}, we fix the fairness parameter $\alpha = 0$ and vary the ambiguity radius $\varepsilon \in [ 0.01, 0.1]$ to examine the effect of ambiguity.
In Figure~\ref{figure: SARA alpha}, we fix the ambiguity radius $\varepsilon$ at the smallest value used in Figures~\ref{figure: SARA epsilon}, i.e., $\varepsilon = 0.01$, to observe the effect of the fairness parameter $\alpha \in [0, 20] \cup \{100\}$.
Thus, both figures start from the same parameter setting: $\varepsilon = 0.01$ and $\alpha = 0$ (the leftmost points in the figures).
We observe a compelling parallel: increasing either the ambiguity radius $\varepsilon$, see Figure~\ref{figure: SARA epsilon}, or the fairness parameter $\alpha$, see Figure~\ref{figure: SARA alpha}, pushes the optimal allocation away from an efficiency-maximization solution (which heavily favors High-Demand users) towards a more equitable distribution among the three user types.  This suggests that while ambiguity and fairness are distinct concepts, increasing the ambiguity parameter $\varepsilon$ can be an effective mechanism for inducing solutions that are structurally more equal.
It is essential to recognize that equality and fairness are distinct concepts; in certain contexts, achieving fairness may necessitate reducing inequality, as noted in~\cite {xinying2023guide}.

\begin{figure}[htbp]
	\centering
	\includegraphics[width=.9\linewidth]{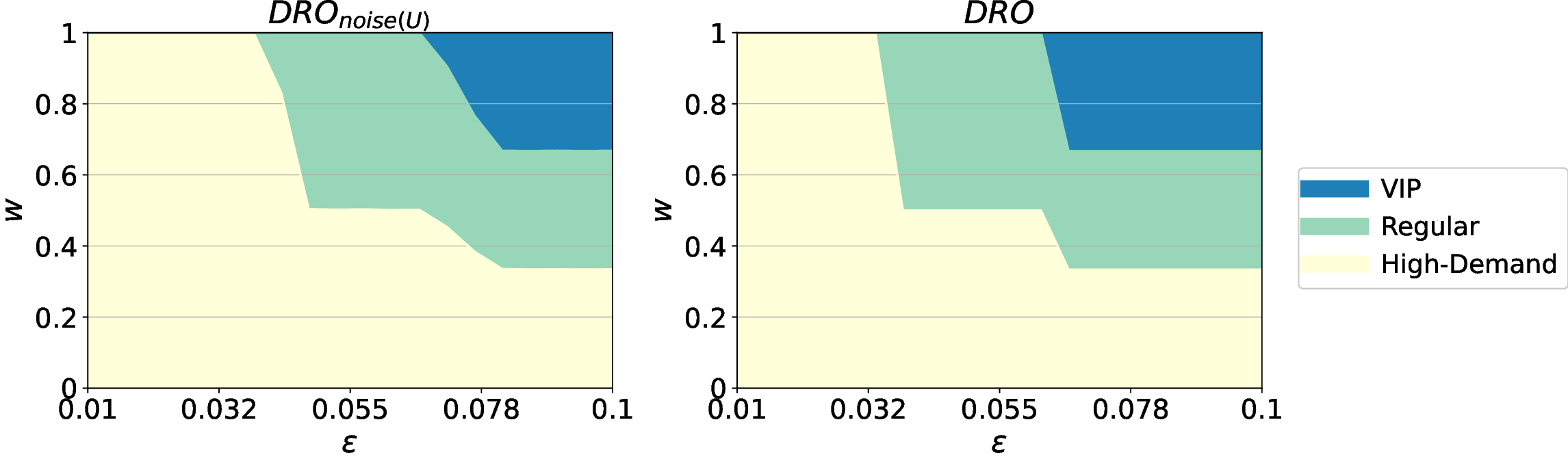}
	\caption{Optimal Allocation Strategies of Base Station 1 Under Varying $\varepsilon$. $(\alpha =0)$ }
	\label{figure: SARA epsilon}
\end{figure}

\begin{figure}[htbp]
	\centering
	\includegraphics[width=.9\linewidth]{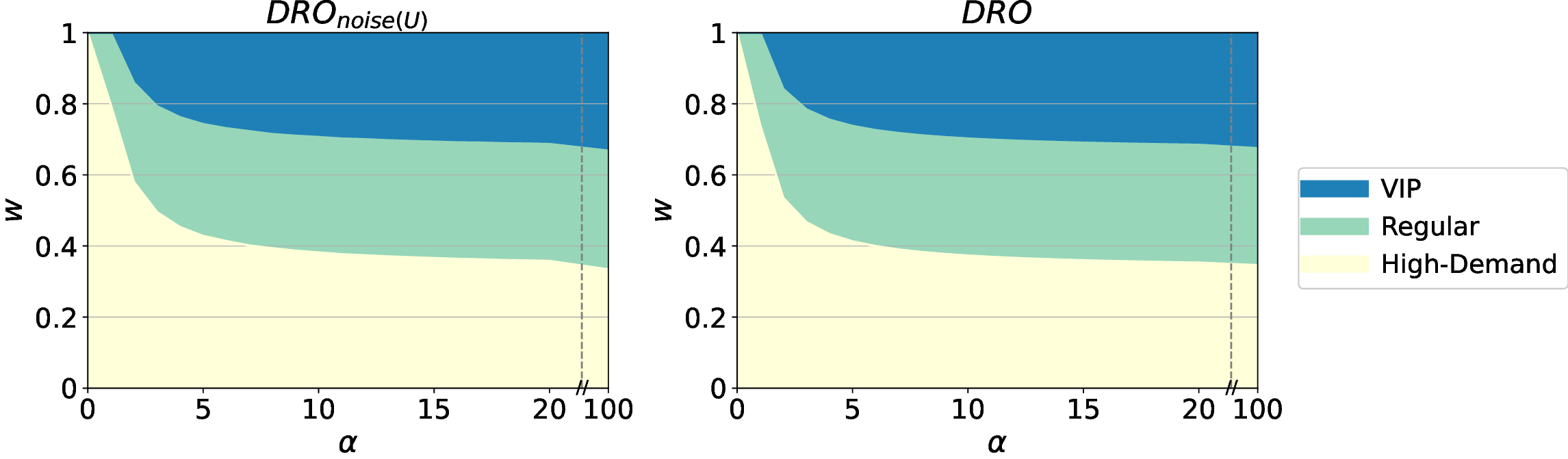}
	\caption{ Optimal Allocation Strategies of Base Station 1 Under Varying $\alpha$. $(\varepsilon = 0.01)$}
	\label{figure: SARA alpha}
\end{figure}

\subsubsection{The Price of Fairness and the Price of Ambiguity.}
In this section, we experiment with the price of ambiguity, as defined in Section~\ref{section: Price of Ambiguity}, and substantiate the dominance of the noisy model over the direct model.
Furthermore, following~\cite{bertsimas2011price}, we also examine the \emph{price of fairness}. In particular, denote the optimal value according to the utilitarian criterion as $\text{SYSTEM}_\text{F}$ and the optimal value under degree of $\alpha$-fairness as $\text{FAIR}( \alpha)$; i.e., for $\alpha = 0$, we have $\text{SYSTEM}_\text{F} = \text{FAIR}(0)$. 
The {price of fairness}, call it $\texttt{POF}( \alpha)$, is defined as
\begin{align*}
\texttt{POF}(\alpha) := \frac{\text{SYSTEM}_\text{F} - \text{FAIR}(\alpha)}{\text{SYSTEM}_\text{F}}.
\end{align*}
Similarly, we analyze $\mathtt{POF}$ for both the direct and noisy models. Note that, unlike the price of ambiguity, which has a unified baseline representing the no-ambiguity case, $\mathtt{POF}$ uses separate baselines $\text{SYSTEM}_{\text{F}}$ for the direct and noisy models, each given by its optimal value at $\alpha = 0$.

\begin{figure}[htbp]
\centering
\includegraphics[width=.8\linewidth]{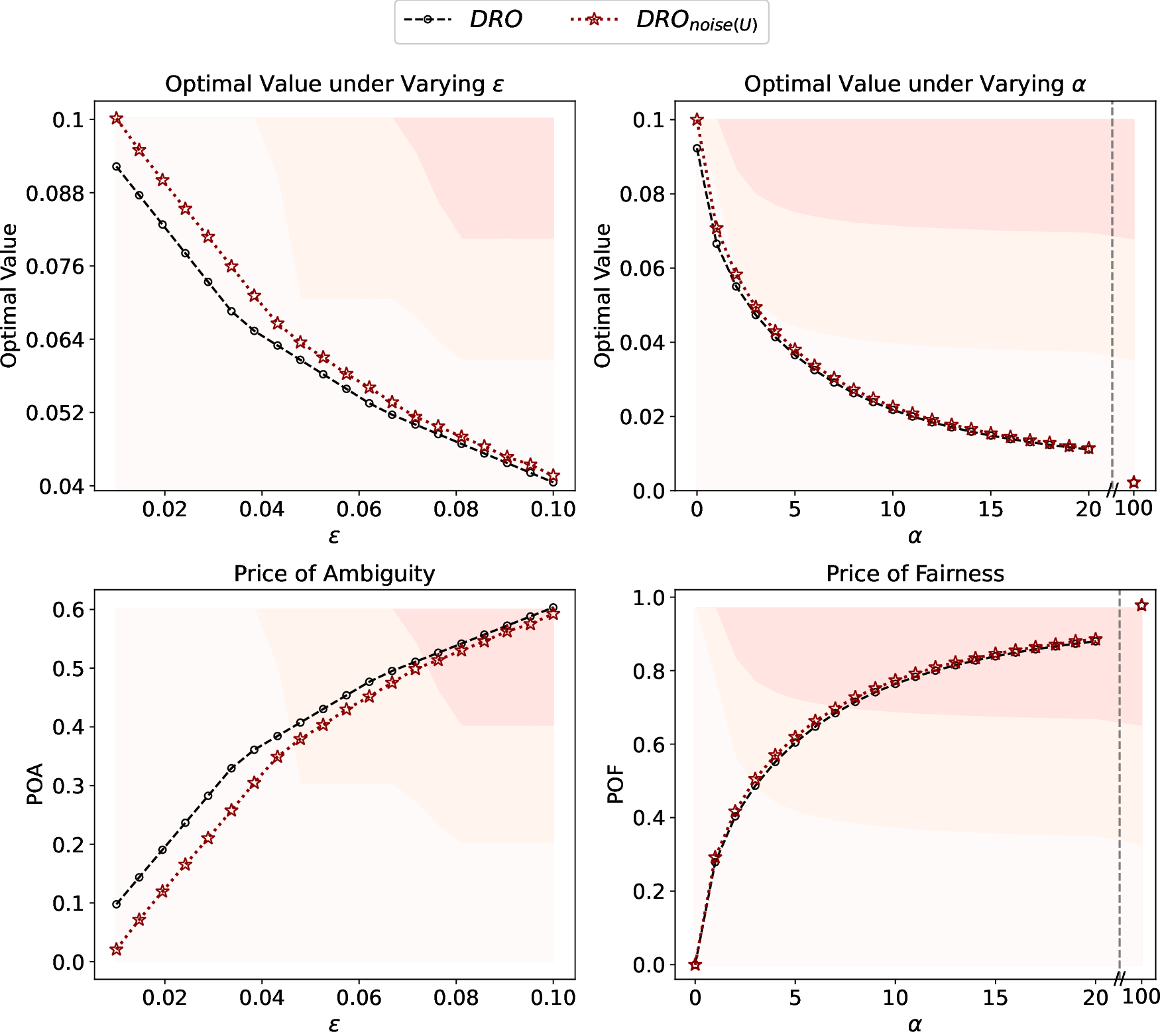}
\caption{Optimal Values and Corresponding \texttt{POF} and \texttt{POA} Under the Same Settings as Figures~\ref{figure: SARA epsilon}--\ref{figure: SARA alpha}. 
	Top Row: Optimal Values; Bottom Row: \texttt{POF} and \texttt{POA}. 
	The Left Two Panels: Vary $\varepsilon$; the Right Two Panels: Vary $\alpha$.}
\label{figure: POF and POA}
\end{figure}

Figure~\ref{figure: POF and POA} shows the optimal values (top row) and the corresponding values of \texttt{POF} and \texttt{POA} (bottom row), under the same settings as Figures~\ref{figure: SARA epsilon}--\ref{figure: SARA alpha}.
In particular, the left two panels show results as a function of $\varepsilon$, while the right two panels show results as a function of $\alpha$. The shaded background regions represent the optimal allocation strategies for $\texttt{DRO}_{\rm noise (U)}$, consistent with those shown in Figures~\ref{figure: SARA epsilon}--\ref{figure: SARA alpha}.
The results empirically confirm Theorem~\ref{theorem: opt values comparison}, demonstrating that the optimal values associated with $\texttt{DRO}_{\rm noise (U)}$ are uniformly higher than those with $\texttt{DRO}$.
Although both parameter variations $\varepsilon$ and $\alpha$, start at the same utility level and both lead to equitable distribution as their values increase, increasing $\varepsilon$ preserves a higher aggregate utility than increasing $\alpha$ for similar levels of equity. 

The bottom panels show similar monotonicity on \texttt{POA} and \texttt{POF}; i.e., increasing $\varepsilon$ leads to a higher value of \texttt{POA}, and increasing $\alpha$ results in a higher value of \texttt{POF}.
Moreover, the value of \texttt{POA} under the direct model is higher than that of the noisy model, consistent with Lemma~\ref{lemma: Dominance of the Price of Ambiguity}.
On the other hand, the noisy model yields slightly higher \texttt{POF} values than the direct model. 

Overall, introducing noise improves aggregate utility and lowers the price of ambiguity, but may result in a higher price of fairness. 
For similar levels of equity, increasing ambiguity $(\varepsilon)$ preserves more aggregate utility than increasing fairness $(\alpha)$.

It is worth noting that we have further tested other noise distribution models beyond the uniform case, for instance, truncated normal, softmax, Bernoulli, binomial, and Poisson distributions (see parameter settings in Appendix~\ref{appendix: noise distribution}).
The trends predicted by Theorem~\ref{theorem: opt values comparison} and Lemma~\ref{lemma: Dominance of the Price of Ambiguity} remain consistent across different noise models.
However, the extent to which noise leads to higher optimal values varies with the magnitude and support of the noise distribution, which in turn affects both \texttt{POA} and \texttt{POF}.

\section{Conclusion}
In this work, we introduced a novel distributionally robust optimization framework tailored for decision-making with noisy data. Our central contribution is a new method for constructing ambiguity sets: rather than building a ball around a naive or pre-processed empirical distribution, we define the set of plausible latent distributions as the pre-image of a Wasserstein ball in the observation space. This pull-back approach provides a principled mechanism to incorporate knowledge of the noise process directly into the robust optimization model.

We have shown that this framework is both theoretically sound and practically viable. We established finite-sample performance guarantees and demonstrated that the solutions to our noisy-data DRO problem converge to the true stochastic optimum as the number of samples increases. 
Furthermore, we derived an equivalent semi-infinite convex program for the DRO problem and discussed the conditions under which it becomes computationally tractable.

Perhaps most significantly, our analysis reveals that noisy data, when modeled correctly, can be a ``blessing in disguise." The resulting noisy-data DRO formulation can be less conservative than its direct counterpart, leading to demonstrably better optimal values and a lower price of ambiguity. Our numerical experiments on a fair resource allocation problem confirm this theoretical insight and further suggest a deep connection between ambiguity and fairness, where robustifying against uncertainty can steer solutions toward more equitable outcomes.

This work opens several avenues for future research. While we assumed a known noise kernel, developing methods to learn or robustify against uncertainty in the noise model itself is an important next step. Extending the framework to handle other forms of data corruption, such as adversarial attacks or heavy-tailed noise, would also be a valuable contribution. 
Moreover, while our analysis focuses on the order-$1$ Wasserstein distance, one avenue for future work is to consider order-$p$ Wasserstein metrics, as well as to more general optimal transport divergences.

\bibliographystyle{apalike}
\bibliography{refs}

\appendix

\section{Technical Proofs}

\subsection{Proofs in Section~\ref{section: Statistical Guarantees for Noisy-Data DRO}} \label{appendix: Statistical Guarantees for Noisy-Data DRO}

\begin{proof}[Proof of Theorem~\ref{theorem: Measure concentration}]
	The proof of part~$(i)$ is standard in the literature, e.g., see~\cite[Theorem 3.4]{mohajerin2018data}.
	
	To prove $(ii)$, we begin by noting that Inequality~\eqref{ineq: measure concentration} implies that
	\begin{align} 
		\mathbb{P}^{\star N} \{\widehat{\mathfrak{X}}^{\star}_N : d(\mathbb{F}^{\star}, \widehat{\mathbb{F}}^{\star}) < \varepsilon \} \geq
		\begin{cases}
			1- c_1 \exp (- c_2 N \varepsilon^{\max\{n,2\}}) &\text{ if } \varepsilon \leq 1;\\
			1- c_1 \exp (-c_2 N \varepsilon^a) &\text{ if } \varepsilon > 1,
		\end{cases}
	\end{align}
	Setting the right-hand side of the inequality above equal to $1-\beta$ and solving for $\varepsilon$ yields
	\begin{align*}
		\varepsilon_N (\beta) =
		\begin{cases}
			\left( \frac{\log(c_1 \beta^{-1})}{c_2 N}\right)^{ \tfrac{1}{ \max \{n,2\} } } &\text{ if } N \geq \frac{\log(c_1 \beta^{-1})}{c_2}; \\
			\left( \frac{\log(c_1 \beta^{-1})}{c_2 N}\right)^{ \tfrac{1}{ a}} & \text{ if } N < \frac{\log(c_1 \beta^{-1})}{c_2} .
		\end{cases}
	\end{align*}

\end{proof}

\begin{proof}[Proof of Corollary \ref{corollary: Concentration Inequality}]
	Let $\varepsilon_N (\beta)$ be chosen such that $\mathbb{P}^{\star N} \{\widehat{\mathfrak{X}}^{\star}_N : d(\mathbb{F}^{\star}, \widehat{\mathbb{F}}^{\star}) \leq \varepsilon_N (\beta) \} \geq  1- \beta$. By the definition of the Wasserstein ball $\mathcal{B}^\star_{ \varepsilon_N (\beta)}(\widehat{\mathbb{F}}^{\star})$, the condition $d(\mathbb{F}^{\star}, \widehat{\mathbb{F}}^{\star}) \leq \varepsilon_N (\beta)$ is equivalent to the condition $\mathbb{F}^\star \in \mathcal{B}^\star_{ \varepsilon_N (\beta)}(\widehat{\mathbb{F}}^{\star})$. This implies that
	$
	\{ \widehat{\mathfrak{X}}^{\star}_N : d(\mathbb{F}^{\star}, \widehat{\mathbb{F}}^{\star}) \leq \varepsilon_N (\beta) \} = \{\widehat{\mathfrak{X}}^{\star}_N : \mathbb{F}^\star \in \mathcal{B}^\star_{ \varepsilon_N (\beta)}(\widehat{\mathbb{F}}^{\star}) \}.
	$ 
	Thus, we obtain 
	\begin{align} \label{eq: concentration inequality for noisy distribution}
		\mathbb{P}^{\star N} \{ \widehat{\mathfrak{X}}^{\star}_N : \mathbb{F}^\star \in \mathcal{B}^\star_{ \varepsilon_N (\beta)}(\widehat{\mathbb{F}}^{\star})\} \geq 1- \beta.
	\end{align}
	Now, using the definition of ambiguity set $\mathcal{B}_{ \varepsilon_N (\beta), \mathbb{O}}(\widehat{\mathbb{F}}^{\star})$ from Equation~\eqref{eq: noisy Wasserstein ball}, we can relate~\eqref{eq: concentration inequality for noisy distribution} to the latent (noiseless) distribution~$\mathbb{F}$. Specifically, note that the condition $ \mathbb{F}^\star \in \mathcal{B}^\star_{ \varepsilon_N (\beta)}(\widehat{\mathbb{F}}^{\star}) $ is equivalent to~$ \mathbb{F} \in \mathcal{B}_{ \varepsilon_N (\beta), \mathbb{O}}(\widehat{\mathbb{F}}^{\star})$. This implies that the events they define are the same set of outcomes: 
	\[
	\{\widehat{\mathfrak{X}}^{\star}_N : \mathbb{F}^\star \in \mathcal{B}^\star_{ \varepsilon_N (\beta)}(\widehat{\mathbb{F}}^{\star})\} = \{\widehat{\mathfrak{X}}^{\star}_N : \mathbb{F} \in \mathcal{B}_{ \varepsilon_N (\beta), \mathbb{O}}(\widehat{\mathbb{F}}^{\star}) \}.
	\]
	Therefore, it follows that
	$
	\mathbb{P}^{\star N} \{\widehat{\mathfrak{X}}^{\star}_N : \mathbb{F} \in \mathcal{B}_{ \varepsilon_N (\beta), \mathbb{O}}(\widehat{\mathbb{F}}^{\star})\} \geq  1-\beta,
	$ 
	and the proof is complete.
\end{proof}

\begin{proof}[Proof of Theorem \ref{theorem: Finite Sample Guarantee under Noisy Data}]
	Let $\varepsilon_N(\beta)$ be chosen such that the concentration inequality from Corollary~\ref{corollary: Concentration Inequality} holds.
	Note that for all $\mathbb{F} \in  \mathcal{B}_{ \varepsilon_N (\beta), \mathbb{O}} (\widehat{\mathbb{F}}^{\star})$, we have
	$$
	\widehat{g}^*_{\text{noise}, N} = \inf_{\mathbb{Q} \in  \mathcal{B}_{ \varepsilon_N (\beta), \mathbb{O} }(\widehat{\mathbb{F}}^{\star})} \mathbb{E}^\mathbb{Q} [U(\widehat{w}^*_{\text{noise}, N}, X)] \leq \mathbb{E}^\mathbb{F} [U(\widehat{w}^*_{\text{noise}, N}, X)].
	$$
	Thus, 
	$
	\{\widehat{\mathfrak{X}}^{\star}_N: \mathbb{F} \in \mathcal{B}_{ \varepsilon_N (\beta), \mathbb{O}}(\widehat{\mathbb{F}}^{\star}) \} \subseteq \{\widehat{\mathfrak{X}}^{\star}_N: \widehat{g}^*_{\text{noise}, N}   \leq \mathbb{E}^\mathbb{F} [U(\widehat{w}^*_{\text{noise}, N}, X)]  \}.
	$
	Taking the probability $\mathbb{P}^{\star N}(\cdot)$ on both sides and using the monotonicity of the probability measure yields.
	\begin{align} \label{ineq: concentration imply out-of-sample guarantee}
		\mathbb{P}^{\star N}\{\widehat{\mathfrak{X}}^{\star}_N:  \mathbb{F} \in \mathcal{B}_{ \varepsilon_N (\beta), \mathbb{O}}(\widehat{\mathbb{F}}^{\star}) \} \leq \mathbb{P}^{\star N}\{ \widehat{\mathfrak{X}}^{\star}_N: \widehat{g}^*_{\text{noise}, N}   \leq \mathbb{E}^\mathbb{F} [U(\widehat{w}^*_{\text{noise}, N}, X)]  \}.
	\end{align}
	Using Corollary~\ref{corollary: Concentration Inequality} that the latent distribution $\mathbb{F}$ lies in $\mathcal{B}_{ \varepsilon_N (\beta), \mathbb{O}}(\widehat{\mathbb{F}}^{\star})$ with probability at least $1 - \beta$; i.e., $	\mathbb{P}^{\star N} \{ \widehat{\mathfrak{X}}^{\star}_N: \mathbb{F} \in \mathcal{B}_{ \varepsilon_N (\beta), \mathbb{O}}(\widehat{\mathbb{F}}^{\star})\} \geq  1-\beta$, in combination with \eqref{ineq: concentration imply out-of-sample guarantee}, we conclude
	\[
	\mathbb{P}^{\star N}\{ \widehat{\mathfrak{X}}^{\star}_N: \widehat{g}^*_{\text{noise}, N}   \leq \mathbb{E}^\mathbb{F} [ U(\widehat{w}^*_{\text{noise}, N}, X)]  \} \geq 1-\beta,
	\] 
	which completes the proof.
\end{proof}

\begin{proof}[Proof of Lemma~\ref{lemma: Noise to Latent Distance}]
	We proceed by contradiction. Assume that there exists a sequence $\{\mathbb{Q}_N\}_{N \geq 1}$ such that the noisy counterpart $\{\mathbb{Q}_N^\star \}_{N \geq 1}$ satisfying
	$\lim_{N \to \infty} d(\mathbb{Q}_N^\star, \mathbb{P}^\star) = 0$, yet
	$\lim_{N \to \infty} d(\mathbb{Q}_N, \mathbb{P}) > 0$. 
	
	The convergence in the Wasserstein metric, $d(\mathbb{Q}_N^\star, \mathbb{P}^\star) \to 0$, implies that 
	$\mathbb{Q}_N^\star$ converges weakly to~$\mathbb{P}^\star$.
	By L\'evy's continuity theorem, see \cite{durrett2019probability}, the convergence is equivalent to the pointwise convergence of their characteristic functions, i.e., $\varphi_{\mathbb{Q}_N^\star}(t)  \to  \varphi_{\mathbb{P}^\star}(t)$ as $N \to \infty$ for all~$t \in \mathbb{R}^n$.
	By the convolution theorem for characteristic functions under the additive noise model, this implies~$\varphi_{\mathbb{Q}_N^\star}(t) = \varphi_{\mathbb{F}_E}(t) \varphi_{\mathbb{Q}_N}(t)$ 
	and
	$\varphi_{\mathbb{P}^\star}(t) = \varphi_{\mathbb{F}_E} (t)\varphi_{\mathbb{P}}(t)$ for all $t \in \mathbb{R}^n$. Hence, for all $t \in \mathbb{R}^n$, we~have
	$$
	\varphi_{\mathbb{F}_E}(t) \varphi_{\mathbb{Q}_N}(t) \to \varphi_{\mathbb{F}_E} (t)\varphi_{\mathbb{P}}(t) \text{ as } N \to \infty.
	$$
	
	Let $D:= \{ t \in \mathbb{R}^n: \varphi_{\mathbb{F}_E}(t) \neq 0 \}.$ By assumption, the complement of $D$, i.e., $D^c = \{ t \in \mathbb{R}^n: \varphi_{\mathbb{F}_E}(t) = 0 \}$ has Lebesgue measure zero, which implies $D$ is dense in $\mathbb{R}^n$.
	For any $t \in D$, we can divide by $\varphi_{\mathbb{F}_E}(t)$ to get pointwise convergence on this dense set: $\varphi_{\mathbb{Q}_N}(t) \to \varphi_{\mathbb{P}}(t)$ all $t \in D$.
	
	Now, we extend this convergence from the dense set $D$ to all of $\mathbb{R}^n$. Since the measures $\{\mathbb{Q}_N\}$ and $\mathbb{P}$ are supported on the same compact set $\mathfrak{X}$, their characteristic functions are uniformly equi-continuous. Specifically, let $R := \sup_{x \in \mathfrak{X}} \|x\|_2 < \infty$. Then, for any of these measures $\mu$ and any~$s,t \in \mathbb{R}^n$, we have
	\begin{align} \label{ineq: uniform Lipschitz bound_initial}
		|\varphi_\mu(t) - \varphi_\mu(s)| 
		&= |\mathbb{E}[e^{i\langle t, X \rangle} - e^{i\langle s, X \rangle}]| \notag \\
		&\leq \mathbb{E}[|e^{i\langle t-s, X \rangle} - 1|] \notag\\
		&\leq \mathbb{E}[|\langle t-s, X \rangle|] \notag \\
		& \leq \|t-s\|_2 \mathbb{E}[\|X\|_2] \leq R \|t-s\|_2, 
	\end{align}
	where the third inequality holds by the fundamental inequality $|e^{i\theta} - 1| \leq |\theta|$ for $\theta = \langle t-s, X \rangle \in \mathbb{R}$. The second last inequality holds by the Cauchy-Schwarz inequality.
	Fix $t \in \mathbb{R}^n$ and $\varepsilon>0$. Since $D$ is dense, we can choose $u \in D$ such that $R\|t - u\|_2 < \tfrac{\varepsilon}{3}$.
	For any $N$, using the triangle inequality, we~have 
	\begin{align*}
		|\varphi_{\mathbb{Q}_N}(t) - \varphi_{\mathbb{P}}(t)| \leq |\varphi_{\mathbb{Q}_N}(t) - \varphi_{\mathbb{Q}_N}(u)| + |\varphi_{\mathbb{Q}_N}(u) - \varphi_{\mathbb{P}}(u)| + |\varphi_{\mathbb{P}}(u) - \varphi_{\mathbb{P}}(t)|.
	\end{align*}
	Using the uniform Lipschitz bound~\eqref{ineq: uniform Lipschitz bound_initial} on the first and third terms gives
	\begin{align*}
		|\varphi_{\mathbb{Q}_N}(t) - \varphi_{\mathbb{P}}(t)| 
		& \leq R\|t-u\|_2 + |\varphi_{\mathbb{Q}_N}(u) - \varphi_{\mathbb{P}}(u)| + R\|t-u\|_2 \\
		& < \frac{2\varepsilon}{3} + |\varphi_{\mathbb{Q}_N}(u) - \varphi_{\mathbb{P}}(u)|.
	\end{align*}
	Since $u \in D$, the middle term $|\varphi_{\mathbb{Q}_N}(u) - \varphi_{\mathbb{P}}(u)| \to 0$ as $N \to \infty$. Thus, for $N$ sufficiently large,~$|\varphi_{\mathbb{Q}_N}(u) - \varphi_{\mathbb{P}}(u)| < \tfrac{\varepsilon}{3}$, which implies  $|\varphi_{\mathbb{Q}_N}(t) - \varphi_{\mathbb{P}}(t)| < \varepsilon$. As $\varepsilon$ was arbitrary, we have shown that~$\varphi_{\mathbb{Q}_N}(t) \to \varphi_{\mathbb{P}}(t)$ for all $t \in \mathbb{R}^n$.
	
	Now, applying L\'evy continuity theorem again, $\mathbb{Q}_N$ converges weakly to $\mathbb{P}$.
	Note that since all measures~$\{\mathbb{Q}_N\}$ and the measure $\mathbb{P}$ are supported on the common compact set $\mathfrak{X}$, weak convergence is equivalent to convergence in the Wasserstein metric. Therefore, we have 
	$
	\lim_{N \to \infty} d(\mathbb{Q}_N, \mathbb{P}) =0,
	$
	which contradicts the assumption that $\lim_{N \to \infty} d(\mathbb{Q}_N, \mathbb{P}) > 0$.
	Hence, the lemma must hold.
\end{proof}

\begin{proof}[Proof of Theorem~\ref{theorem: Asymptotic Consistency}]
	To prove part~$(i)$, we take $\widehat{w}^*_{\text{noise}, N} \in \mathcal{W}$. From the formulation of DRO, we know that for any $N \in \mathbb{N}$, we have
	$ \widehat{g}^*_{\text{noise}, N} \leq \mathbb{E}^{\mathbb{F}} [U(\langle \widehat{w}^*_{\text{noise}, N}, X \rangle)] \leq g_{SO}^*$ for $\mathbb{F} \in \mathcal{B}_{\varepsilon_N (\beta_N), \mathbb{O}}(\widehat{\mathbb{F}}^\star)$.
	It follows that
	\begin{align*}
		\mathbb{P}^{\star N} \{ \widehat{\mathfrak{X}}^{\star}_N: \widehat{g}^*_{\text{noise}, N} \leq \mathbb{E}^{\mathbb{F}} [U( \langle \widehat{w}^*_{\text{noise}, N}, X \rangle)] \leq g^*_{SO}\} 
		&\geq \mathbb{P}^{\star N} \{ \widehat{\mathfrak{X}}^{\star}_N: \mathbb{F} \in  \mathcal{B}_{ \varepsilon_N (\beta_N), \mathbb{O}} (\widehat{\mathbb{F}}^\star)\}.
	\end{align*}
	Applying Corollary~\ref{corollary: Concentration Inequality} yields that $\mathbb{P}^{\star N} \{ \widehat{\mathfrak{X}}^{\star}_N : \mathbb{F} \in  \mathcal{B}_{ \varepsilon_N (\beta_N), \mathbb{O}} (\widehat{\mathbb{F}}^\star)\} \geq 1 - \beta_N $, which implies that  
	\begin{align*}
		\mathbb{P}^{\star N} \{ \widehat{\mathfrak{X}}^{\star}_N: \widehat{g}^*_{\text{noise}, N} \leq \mathbb{E}^{\mathbb{F}} [U( \langle \widehat{w}^*_{\text{noise}, N}, X \rangle)] \leq g^*_{SO}\} 
		&\geq 1 - \beta_N.
	\end{align*}
	Define the event 
	$
	A_N := \{ \widehat{\mathfrak{X}}^{\star}_N: \widehat{g}^*_{\text{noise}, N} \leq \mathbb{E}^{\mathbb{F}} [U( \langle \widehat{w}^*_{\text{noise}, N}, X \rangle)] \leq g^*_{SO}\}
	$, 
	and let $B_N := A_N^c$ denote its complement.
	Then
	$\mathbb{P}^{\star N} (B_N) \leq \beta_N$.
	Since $\sum_{N=1}^{\infty} \beta_N < \infty$, the Borel–Cantelli Lemma, e.g., see \cite{durrett2019probability}, implies that the probability of events $B_N$ occurring infinitely often is zero, i.e.,
	$
	\mathbb{P}^{\star \infty} (\limsup_{N\to \infty} B_N) = 0,
	$
	which implies
	\begin{align*}
		&\mathbb{P}^{\star \infty} \{ \widehat{\mathfrak{X}}^{\star}_N: A_N \text{ holds for all sufficiently large $N$}\}\\
		&=
		\mathbb{P}^{\star \infty} \{ \widehat{\mathfrak{X}}^{\star}_N: \widehat{g}^*_{\text{noise}, N} \leq \mathbb{E}^{\mathbb{F}} [U( \langle \widehat{w}^*_{\text{noise}, N}, X \rangle)] \leq g^*_{SO} \text{ for all sufficiently large $N$}\} = 1.
	\end{align*}
	That is, we have $\limsup_{N \to \infty} \widehat{g}^*_{\text{noise}, N} \leq g^*_{SO}$ with probability one.
	
	To complete the proof of $(i)$, it remains to show that $\liminf_{N \to \infty} \widehat{g}^*_{\text{noise}, N} \geq g^*_{SO}$ almost surely.
	Fix an arbitrary $\delta >0$, and let $w_{\delta} \in \mathcal{W}$ be a $\delta$-optimal solution for $\sup_{w \in \mathcal{W}} \mathbb{E}^{\mathbb{F}} [U( \langle w, X \rangle )]$, i.e.,~ $\mathbb{E}^{\mathbb{F}} [U(\langle w_{\delta}, X \rangle)] \geq g_{SO}^* - \delta$.
	For each $N \in \mathbb{N}$, let $\widehat{\mathbb{Q}}_N \in \mathcal{B}_{ \varepsilon_N (\beta_N), \mathbb{O}} (\widehat{\mathbb{F}}^\star)$ be a distribution that is $\delta$-optimal corresponding to $w_{\delta}$, satisfying $\inf_{\mathbb{Q} \in \mathcal{B}_{ \varepsilon_N (\beta_N), \mathbb{O} } (\widehat{\mathbb{F}}^\star)} \mathbb{E}^{\mathbb{Q}} [U( \langle w_{\delta}, X \rangle)] \geq \mathbb{E}^{\widehat{\mathbb{Q}}_N} [U( \langle w_{\delta}, X \rangle)] - \delta$. 
	Then, we observe that
	\begin{align}
		\liminf_{N \to \infty} \widehat{g}^*_{\text{noise}, N} 
		&\geq \liminf_{N \to \infty} \inf_{\mathbb{Q} \in \mathcal{B}_{ \varepsilon_N (\beta_N), \mathbb{O} } (\widehat{\mathbb{F}}^\star)} \mathbb{E}^{\mathbb{Q}} [U( \langle w_{\delta}, X \rangle)] \notag\\
		&\geq \liminf_{N \to \infty} \mathbb{E}^{\widehat{\mathbb{Q}}_N} [U( \langle w_{\delta}, X \rangle)] - \delta \notag\\
		&=
		\liminf_{N \to \infty} \left( \mathbb{E}^{\mathbb{F}} [ U( \langle w_{\delta}, X \rangle) ] + \left( \mathbb{E}^{\widehat{\mathbb{Q}}_N} [ U( \langle w_{\delta}, X \rangle) ] - \mathbb{E}^\mathbb{F} [ U( \langle w_{\delta}, X \rangle) ] \right)  \right) - \delta. \label{eq: before Kantorovich-Rubenstein duality}
	\end{align}
	
	Since $U(\langle w_{\delta}, x \rangle)$ is $L$-Lipschitz continuous on $\mathfrak{X}$, by the Kantorovich-Rubinstein duality, we have
	$$
	\mathbb{E}^\mathbb{F} [ U( \langle w_{\delta}, X \rangle) ] - \mathbb{E}^{\widehat{\mathbb{Q}}_N} [ U( \langle w_{\delta}, X \rangle) ] 
	\leq 
	\sup_{f \in \mathcal{L}} \left(  \mathbb{E}^\mathbb{F} [ U( \langle w_{\delta}, X \rangle) ] - \mathbb{E}^{\widehat{\mathbb{Q}}_N} [ U( \langle w_{\delta}, X \rangle) ] \right)
	=  L \cdot d(\widehat{\mathbb{Q}}_N, \mathbb{F}),
	$$
	where  $\mathcal{L}$ is the set of all $L$-Lipschitz functions satisfying $|f(x) - f(x')| \leq L\| x - x'\|$ for all $x, x' \in \mathfrak{X}$.
	Applying the same bound for $-U$, we can further bound Inequality~\eqref{eq: before Kantorovich-Rubenstein duality} as follows:
	\begin{align}
		& \liminf_{N \to \infty} \left( \mathbb{E}^{\mathbb{F}} [ U( \langle w_{\delta}, X \rangle) ] + \left( \mathbb{E}^{\widehat{\mathbb{Q}}_N} [ U( \langle w_{\delta}, X \rangle) ] - \mathbb{E}^\mathbb{F} [ U( \langle w_{\delta}, X \rangle) ] \right) \right) - \delta \notag\\
		& \qquad \geq 
		\liminf_{N \to \infty} \left( \mathbb{E}^{\mathbb{F}} [ U( \langle w_{\delta}, X \rangle) ] -  L\cdot d (\widehat{\mathbb{Q}}_N, \mathbb{F}) \right)- \delta \notag \\
		& \qquad =    \mathbb{E}^{\mathbb{F}} [ U( \langle w_{\delta}, X \rangle) ] -  \liminf_{N \to \infty} L\cdot d (\widehat{\mathbb{Q}}_N, \mathbb{F})  - \delta. \label{eq: Kantorovich-Rubenstein duality}
	\end{align}
	
	We now prove the almost sure convergence of the distributions in noisy space, namely, for $\widehat{\mathbb{Q}}_N^\star \in \mathcal{B}_{ \varepsilon_N (\beta_N), \mathbb{O}} (\widehat{\mathbb{F}}^\star)$, we shall establish
	$
	\mathbb{P}^{\star \infty} \{ \widehat{\mathfrak{X}}^{\star}_N : \lim_{N \to \infty} d(\widehat{\mathbb{Q}}_N^\star, \mathbb{F}^\star) = 0
	\} =  1.
	$ 
	Since $\mathfrak{X} \subseteq \mathfrak{X}^{\star} \subseteq \mathbb{R}^n$, applying the triangle inequality for the Wasserstein distance, we have:
	\begin{align*}
		d(\widehat{\mathbb{Q}}_N^\star, \mathbb{F}^\star) 
		&\leq d(\widehat{\mathbb{Q}}_N^\star, \widehat{\mathbb{F}}^\star) + d(\widehat{\mathbb{F}}^\star, \mathbb{F^\star})  \\
		&\leq \varepsilon_N (\beta_N) + d(\widehat{\mathbb{F}}^\star, \mathbb{F^\star}) 
	\end{align*}
	where the second inequality holds since $\widehat{\mathbb{Q}}_N^\star \in \mathcal{B}_{ \varepsilon_N (\beta_N), \mathbb{O}} (\widehat{\mathbb{F}}^\star)$ implies $d(\widehat{\mathbb{Q}}_N^\star, \widehat{\mathbb{F}}^\star) \leq \varepsilon_N (\beta_N)$.
	Moreover, from Corollary~\ref{corollary: Concentration Inequality}, we know that
	$\mathbb{P}^{\star N} \{ \widehat{\mathfrak{X}}^{\star}_N :  d(\widehat{\mathbb{F}}^{\star}, \mathbb{F}^\star) \leq \varepsilon_N (\beta_N)\} \geq  1-\beta_N$, which is equivalent to 
	$$
	\mathbb{P}^{\star N} \{  \widehat{\mathfrak{X}}^{\star}_N : \varepsilon_N (\beta_N) + d(\widehat{\mathbb{F}}^\star, \mathbb{F}^\star) \leq 2 \varepsilon_N (\beta_N) \} \geq  1-\beta_N.
	$$
	Define the event $C_N := \{  \widehat{\mathfrak{X}}^{\star}_N :  \varepsilon_N (\beta_N)  + d(\widehat{\mathbb{F}}^\star, \mathbb{F}^\star) \leq 2 \varepsilon_N (\beta_N) \}$, and let $D_N := C_N^c$ be its complement. 
	Then, $\mathbb{P}^{\star N} (D_N) \leq \beta_N$.
	Since $\sum_{N=1}^{\infty} \beta_N < \infty$, the Borel–Cantelli Lemma implies that $\mathbb{P}^{\star \infty} (\limsup_{N \to \infty} D_N ) = 0$. Hence, we have
	$$
	\mathbb{P}^{\star \infty} \{ \widehat{\mathfrak{X}}^{\star}_N : \varepsilon_N (\beta_N)  + d(\widehat{\mathbb{F}}^\star, \mathbb{F}^\star) 
	\leq 2 \varepsilon_N (\beta_N) \text{ for all sufficiently large } N \} =  1.
	$$
	Because $\lim_{N \to \infty} \varepsilon_N (\beta_N) = 0$, it follows that 
	$\lim_{N \to \infty} d(\widehat{\mathbb{Q}}_N^\star, \mathbb{F}^\star) = 0$ almost surely, i.e.,
	$$
	\mathbb{P}^{\star \infty} \{ \widehat{\mathfrak{X}}^{\star}_N : \lim_{N \to \infty} d(\widehat{\mathbb{Q}}_N^\star, \mathbb{F}^\star) = 0
	\} =  1.
	$$
	We now transfer this convergence result from the noisy space to the latent space.
	We have already established that for almost every sample path, $\lim_{N \to \infty} d(\widehat{\mathbb{Q}}_N^\star, \mathbb{F}^\star) = 0$.
	Consider any such sample path. We will show  that this implies $\lim_{N \to \infty} d(\widehat{\mathbb{Q}}_N, \mathbb{F}) = 0$. Assume the contrary, i.e., that $d(\widehat{\mathbb{Q}}_N, \mathbb{F})$ does not converge to 0. Then, there must exist an $\eta > 0$ and a subsequence $\{N_k\}_{k=1}^\infty$ such that $d(\widehat{\mathbb{Q}}_{N_k}, \mathbb{F}) > \eta$ for all $k$.
	For this same subsequence, however, we still have $\lim_{k \to \infty} d(\widehat{\mathbb{Q}}_{N_k}^\star, \mathbb{F}^\star) = 0$. This presents a sequence of distributions $\{\widehat{\mathbb{Q}}_{N_k}\}$ for which~$\lim_{k \to \infty} d(\widehat{\mathbb{Q}}_{N_k}^\star, \mathbb{F}^\star) = 0$ but $\lim_{k \to \infty} d(\widehat{\mathbb{Q}}_{N_k}, \mathbb{F}) \neq 0$. This is a direct contradiction to Lemma~\ref{lemma: Noise to Latent Distance}.
	Thus, our assumption must be false. It follows that for any sample path where the noisy distributions converge, the latent distributions must also converge. Since the set of such sample paths has a probability of 1, we conclude that:
	$$
	\mathbb{P}^{\star \infty} \{ \widehat{\mathfrak{X}}^{\star}_N : \lim_{N \to \infty} d(\widehat{\mathbb{Q}}_N, \mathbb{F}) = 0\} =  1.
	$$
	
	Returning to Equation~\eqref{eq: Kantorovich-Rubenstein duality}, we conclude that
	\begin{align*}
		\liminf_{N \to \infty} \widehat{g}^*_{\text{noise}, N} 
		&\geq \mathbb{E}^{\mathbb{F}} [ U( \langle w_{\delta}, X \rangle) ] - \delta, \quad \mathbb{P}^{\star, \infty}\text{-almost surely}\\
		& \geq g^*_{SO} - 2\delta .
	\end{align*}
	The last inequality holds because $w_\delta$ was chosen to be $\delta$-suboptimal for $g^*_{SO}$.
	Since~$\delta > 0$ was arbitrary, we conclude
	$\liminf_{N \to \infty} \widehat{g}^*_{\text{noise}, N} \geq g^*_{SO}$.
	Combined with the earlier bound~$\limsup_{N \to \infty} \widehat{g}^*_{\text{noise}, N} \leq g^*_{SO}$, we obtain 
	$\lim_{N \to \infty} \widehat{g}_{\text{noise}, N}^* = g^*_{SO}$ almost surely.

	To prove part~$(ii)$, we fix an arbitrary realization of $\{\widehat{X}_j\}_{j=1}^N$ such that for all sufficiently large~$N$,
	$$
	\widehat{g}_{\text{noise}, N}^* \leq \mathbb{E}^\mathbb{F} [U (\langle \widehat{w}_{\text{noise}, N}^*, X \rangle) ] \leq g_{SO}^*.
	$$
	Taking the limit superior on the above inequalities and noting that, from part $(i)$, we have $\lim_{N \to \infty} \widehat{g}_{\text{noise}, N}^* = g_{SO}^*$ almost surely. Therefore, it follows that
	\begin{align}
		g_{SO}^* 
		= \lim_{N \to \infty} \widehat{g}_{\text{noise}, N}^* 
		&\leq
		\limsup_{N \to \infty} \mathbb{E}^\mathbb{F} [U (\langle \widehat{w}_{\text{noise}, N}^*, X \rangle) ] \notag \\
		&\leq 
		\mathbb{E}^\mathbb{F} [\limsup_{N \to \infty} U (\langle \widehat{w}_{\text{noise}, N}, X \rangle) ],  \label{ineq: g_SO bound}
	\end{align}
	where the last inequality is obtained by reverse Fatou's lemma, since $U$ is bounded by a finite $\mathfrak{X}$ for fixed $w$.
	Since~$\mathcal{W}$ is compact, the sequence $\{\widehat{w}^*_{\text{noise} , N}\}_{N \in \mathbb{N}}$ has an accumulation point $\bar{w} \in \mathcal{W}$.
	By selecting a convergent subsequence, we can assume without loss of generality that $\lim_{N \to \infty} \widehat{w}_{\text{noise}, N} = \bar{w} \in \mathcal{W}$.
	Since $U$ is Lipschitz continuous, it is also continuous and hence upper semicontinuous. Therefore, by the definition of upper semicontinuity, we can further bound~\eqref{ineq: g_SO bound} and obtain
	\begin{align*}
		\mathbb{E}^\mathbb{F} [\limsup_{N \to \infty} U (\langle \widehat{w}_{\text{noise}, N}, X \rangle) ] 
		&\leq 
		\mathbb{E}^\mathbb{F}[U(\langle \bar{w}, X\rangle)] \\
		&\leq g_{SO}^*,
	\end{align*}
	the last inequality holds since $\bar{w} \in \mathcal{W}$.
	Putting everything together, we obtain 
	$
	\mathbb{E}^\mathbb{F}[U(\langle \bar{w}, X \rangle)] 
	= g_{SO}^*.
	$
	Thus, 
	$
	\mathbb{P}^{\star \infty} \{ \widehat{\mathfrak{X}}^{\star}_N : \mathbb{E}^\mathbb{F}[U(\langle \bar{w}, X\rangle)] 
	= g_{SO}^*
	\} =  1,
	$
	which implies that $\bar{w}$ is almost surely an optimal solution for the stochastic optimization problem with probability one.
\end{proof}

\subsection{Proofs in Section~\ref{section: main results}} \label{appendix: proof in Section: main results}

\begin{proof}[Proof of Lemma~\ref{lemma: An Equivalent Representation of Noisy-Data DRO}]  
	Fix $\varepsilon > 0$. We begin by writing the inner worst-case expectation of the noisy-data DRO problem~\eqref{eq: noise DRO}~as
	\begin{align*}
		\inf_{ \mathbb{F} \in \mathcal{B}_{\varepsilon, \mathbb{O}} (\widehat{\mathbb{F}}^{\star})} \mathbb{E}^{\mathbb{F}} \left[  U \left( \langle w, X \rangle \right) \right] 
		&= \begin{cases} 
			\displaystyle\inf_{\Pi, \mathbb{F}} \int_{x \in \mathfrak{X}} U \left( \langle w, x \rangle \right) d \mathbb{F}(x) \\
			{\rm s.t.} \; \displaystyle\mathbb{E}^{\Pi} [\|X^\star -\widehat{X}^\star\| ] \leq \varepsilon; \\
			\qquad  \Pi \in \mathcal{C}(T_{\mathbb{O}}(\mathbb{F}), \widehat{\mathbb{F}}^\star)
		\end{cases}\\
		&=\begin{cases} 
			\displaystyle\inf_{\Pi, \mathbb{F}} \int_{x \in \mathfrak{X}} U \left( \langle w, x \rangle \right) d\mathbb{F}(x) \\
			{\rm s.t.} \; \displaystyle\int_{x^{\star} \in \mathfrak{X}^\star, \widehat{x}^\star \in \mathfrak{X}^\star} \|x^\star - \widehat{x}^\star\|\, d \Pi(\widehat{x}^\star, x^\star) \leq \varepsilon,\\
			\qquad  \Pi \in \mathcal{C}(T_{\mathbb{O}}(\mathbb{F}), \widehat{\mathbb{F}}^\star),
		\end{cases}
	\end{align*}
	where $\Pi$ represents any valid joint distribution (coupling) of a random variable $X^\star \sim T_{\mathbb{O}}(\mathbb{F})$ and $\widehat{X}^\star \sim \widehat{\mathbb{F}}^\star$.
	Since the target measure $\widehat{\mathbb{F}}^\star$ is empirical, any such coupling $\Pi$ can be decomposed as
	$
	\Pi = \frac{1}{N}\sum_{j=1}^{N} \delta_{\widehat{x}_j^\star} \otimes G_j^\star,
	$
	where $\{G_j^\star\}_{j=1}^N$ are probability measures on $\mathfrak{X}^\star$ that satisfy the marginal constraint $\frac{1}{N}\sum_{j=1}^{N} G_j^\star = T_{\mathbb{O}}(\mathbb{F})$. The problem is to find the infimum over all valid latent measures $\mathbb{F}$ and all valid component measures ${G_j^\star}$.

	The objective function only depends on $\mathbb{F}$. Any latent distribution $\mathbb{F}$ can be represented as a mixture $\mathbb{F} = \frac{1}{N} \sum_{j=1}^N \alpha_j$ for some (not necessarily unique) component latent measures $\{\alpha_j\}_{j=1}^N \subset \mathcal{M}(\mathfrak{X})$. By linearity of the noise operator $T_\mathbb{O}$, the marginal constraint becomes:
	$$
	\frac{1}{N}\sum_{j=1}^N G_j^\star = T_{\mathbb{O}}(\mathbb{F}) = T_{\mathbb{O}}\left(\frac{1}{N}\sum_{j=1}^N \alpha_j\right) = \frac{1}{N}\sum_{j=1}^N T_{\mathbb{O}}(\alpha_j).
	$$
	A sufficient way to satisfy this is to set $G_j^\star = T_{\mathbb{O}}(\alpha_j)$ for each $j$. 
	This allows us to reparameterize the optimization problem in terms of an unconstrained choice of component latent measures $\{\alpha_j\}_{j=1}^N \subset \mathcal{M}(\mathfrak{X})$. The objective function becomes $\mathbb{E}^{\mathbb{F}}[U\langle w, x \rangle] =  \frac{1}{N} \sum_{j=1}^{N} \int_{x \in \mathfrak{X}} U \left( \langle w, x \rangle \right) d \alpha_j(x)$, and the cost constraint becomes:
	$$
	\frac{1}{N}\sum_{j=1}^{N} \int_{x^{\star} \in \mathfrak{X}^\star} \|x^\star - \widehat{x}_j^\star\| dG_j^\star(x^\star) = \frac{1}{N}\sum_{j=1}^{N} \int_{x \in \mathfrak{X}} \left( \int_{x^\star \in \mathfrak{X}^\star} \|x^{\star} - \widehat{x}_j^\star\| d\mathbb{O}(x^{\star} \mid x) \right) d\alpha_j(x) \leq \varepsilon
	$$
	Therefore, the inner problem is thus equivalent to:
	\begin{align*}
		\inf_{\mathbb{F} \in \mathcal{B}_{\varepsilon, \mathbb{O}} (\widehat{\mathbb{F}}^{\star})}  \mathbb{E}^{\mathbb{F}} [ U \left( \langle w, X \rangle \right) ]
		&=  \begin{cases} 
			\displaystyle\inf_{\alpha_j \in  \mathcal{M}(\mathfrak{X}), \forall j} \frac{1}{N} \sum_{j=1}^{N} \int_{x \in \mathfrak{X}} U \left( \langle w, x \rangle \right) d \alpha_j(x)\\
			{\rm s.t.} \; \displaystyle \frac{1}{N} \sum_{j=1}^{N} \int_{x \in \mathfrak{X}} \left( \int_{x^\star \in \mathfrak{X}^\star} \|x^{\star} - \widehat{x}_j^\star\| d\mathbb{O}(x^{\star} \mid x) \right) d\alpha_j(x) \leq \varepsilon.
		\end{cases}
	\end{align*}
	Introducing a Lagrange multiplier $\lambda \geq 0$, we obtain the Lagrangian dual problem:
	\begin{align}
		\inf_{\mathbb{F} \in \mathcal{B}_{\varepsilon, \mathbb{O}} (\widehat{\mathbb{F}}^{\star})}  \mathbb{E}^{\mathbb{F}} [ U \left( \langle w, X \rangle \right) ] 
		&= \inf_{\alpha_j \in \mathcal{M}(\mathfrak{X}), \forall j} \sup_{\lambda \geq 0}  \mathcal{L}( \{ \alpha_j\}, \lambda) 
	\end{align}
	where $\mathcal{L}( \{ \alpha_j\}, \lambda) := - \lambda \varepsilon + \frac{1}{N} \sum_{j=1}^{N} \int_{x \in \mathfrak{X}} \left( U \left( \langle w, x \rangle \right) + \lambda \int_{x^\star \in \mathfrak{X}^\star} \|x^{\star}-\widehat{x}_j^\star\| d\mathbb{O}(x^{\star} \mid x) \right) d\alpha_j(x)$.
	Note that for each fixed $\lambda$, $\mathcal{L}(\cdot )$ is continuous and linear in $\{\alpha_j \}$, hence is convex and lower semicontinuous; for each fixed $\{\alpha_j \}$, it is affine, hence concave and continuous, in $\lambda$. The domain for the measures $\mathcal{M}(\mathfrak{X})$ is convex and is compact in the weak* topology because the underlying space $\mathfrak{X}$ is compact, see \cite{aliprantis2006infinite}, and the dual domain $[0, \infty)$ is convex. Therefore, Sion's minimax theorem applies and yields
	\begin{align}
		&\inf_{\alpha_j  \in \mathcal{M}(\mathfrak{X}), \forall j} \sup_{\lambda \geq 0}  \mathcal{L}( \{ \alpha_j \}, \lambda)  \notag \\
		&\qquad =  \sup_{\lambda \geq 0}  \inf_{\alpha_j  \in \mathcal{M}(\mathfrak{X}), \forall j}  \mathcal{L}( \{ \alpha_j \}, \lambda) \notag \\
		&\qquad =
		\sup_{\lambda \geq 0} \inf_{\alpha_j  \in \mathcal{M}(\mathfrak{X}), \forall j} - \lambda \varepsilon + \frac{1}{N} \sum_{j=1}^{N}\int_{x \in \mathfrak{X}} \bigg( U \left( \langle w, x \rangle \right) + \lambda \int_{x^\star \in \mathfrak{X}^\star} \|x^{\star}-\widehat{x}_j^\star\| d\mathbb{O}(x^{\star} \mid x) \bigg) d\alpha_j (x) \notag \\
		&\qquad = 
		\sup_{\lambda \geq 0} - \lambda \varepsilon + \inf_{\alpha_j  \in \mathcal{M}(\mathfrak{X}), \forall j} \frac{1}{N} \sum_{j=1}^{N}\int_{x \in \mathfrak{X}} \bigg( U \left( \langle w, x \rangle \right) + \lambda \int_{x^\star \in \mathfrak{X}^\star}  \|x^{\star}-\widehat{x}_j^\star\|  d\mathbb{O}(x^{\star} \mid x) \bigg) d\alpha_j (x). \label{eq: minimax (noise)}
	\end{align}
	Since the optimization over $\{\alpha_j\}$ is unconstrained (beyond each being a probability measure) and the objective is a sum, the problem~\eqref{eq: minimax (noise)} decouples:
	$$
	\sup_{\lambda \geq 0} - \lambda \varepsilon + \frac{1}{N} \sum_{j=1}^{N}  \inf_{\alpha_j  \in \mathcal{M}(\mathfrak{X})} \int_{x \in \mathfrak{X}} \bigg( U \left( \langle w, x \rangle \right) + \lambda \int_{x^\star \in \mathfrak{X}^\star}  \|x^{\star}-\widehat{x}_j^\star\|  d\mathbb{O}(x^{\star} \mid x) \bigg) d\alpha_j (x). \nonumber
	$$
	For each $j$, the inner infimum is a linear functional over the space of probability measures. The infimum is therefore attained at a Dirac delta measure concentrated on the point $x \in \mathfrak{X}$ that minimizes the integrand. Thus, we have:
	\begin{align*}
		\inf_{\mathbb{F} \in \mathcal{B}_{\varepsilon, \mathbb{O}} (\widehat{\mathbb{F}}^{\star})}  \mathbb{E}^{\mathbb{F}} [ U \left( \langle w, X \rangle \right) ] 
		&=
		\sup_{\lambda \geq 0} -  \lambda \varepsilon +  \frac{1}{N} \sum_{j=1}^{N}  \inf_{x \in \mathfrak{X} } \bigg( U \left( \langle w, x \rangle \right) + \lambda \int_{x^\star \in \mathfrak{X}^\star} \|x^{\star}-\widehat{x}_j^\star\| d\mathbb{O}(x^{\star} \mid x) \bigg).
	\end{align*}
	Introducing epigraph variables $a_j$, for $j =1, \ldots, N$, we reformulate the problem as 
	\begin{align}
		\inf_{\mathbb{F} \in \mathcal{B}_{\varepsilon, \mathbb{O}} (\widehat{\mathbb{F}}^{\star})}  \mathbb{E}^{\mathbb{F}} [ U \left( \langle w, X \rangle \right) ]\nonumber 
		&=
		\begin{cases}
			\displaystyle\sup_{\lambda \geq 0, a_j, \forall j} - \lambda \varepsilon +  \frac{1}{N} \sum_{j=1}^{N} a_j\\
			{\rm s.t.} \;  \displaystyle\inf_{x \in \mathfrak{X}} U \left( \langle w, x \rangle \right) + \lambda \int_{x^\star \in \mathfrak{X}^\star}  \|x^{\star}-\widehat{x}_j^\star\| d\mathbb{O}(x^{\star}\mid x)  \geq a_j,\quad j = 1, \dots, N.
		\end{cases}\nonumber
	\end{align}
	
	Restoring the outer maximization over $w \in \mathcal{W}$ in DRO problem~\eqref{eq: noise DRO} yields the desired equivalent representation:
	\begin{align*}
		\begin{cases}
			\displaystyle\sup_{w \in \mathcal{W}, \lambda \geq 0, a_j , \forall j} - \lambda \varepsilon +  \frac{1}{N} \sum_{j=1}^{N} a_j\\
			{\rm s.t.} \;  \displaystyle\inf_{x \in \mathfrak{X}}\displaystyle U \left( \langle w, x \rangle \right) + \lambda \int_{x^\star \in \mathfrak{X}^\star}  \|x^{\star}-\widehat{x}_j^\star\| d\mathbb{O}(x^{\star} \mid x)  \geq a_j,\quad j = 1, \dots, N.
		\end{cases}
	\end{align*}
	We verify that the problem above is a convex program.
	First, note that for any $\varepsilon >0$, the objective function is affine in $\lambda$ and $a_j$ for all $j = 1, \ldots, N$.
	Next, consider the constraint. 
	Fixed $x\in \mathfrak{X}$,
	the utility function~$U \left( \langle w, x \rangle \right)$ is concave in $w$, 
	and the term~$\int_{x^\star \in \mathfrak{X}^\star}  \|x^{\star}-\widehat{x}_j^\star\| d\mathbb{O}(x^{\star} \mid x)$ is constant. 
	Hence, the inner expression $U \left( \langle w, x \rangle \right) + \lambda \int_{x^\star \in \mathfrak{X}^\star}  \|x^{\star}-\widehat{x}_j^\star\| d\mathbb{O}(x^{\star} \mid x)$ is affine in $\lambda$ and concave in $w$.
	The pointwise infimum of a family of concave functions, with index by $x \in \mathfrak{X}$, is a concave function. Thus, the left-hand side of the constraint remains concave in $w$ and $\lambda$, while the right-hand side is linear in~$a_j$.
	Since the feasible set is defined by a semi-infinite family of convex inequalities and the objective is affine, the problem is a convex program.
	This concludes the proof.
\end{proof}


\begin{proof}[Proof of Theorem \ref{theorem: Sensitivity of the Value Function}] 
	The value function of the noisy-data DRO problem is given by
	$$
	g_{\rm noise}^*(\varepsilon) = \sup_{w \in \mathcal{W}} \left( \inf_{\mathbb{F} \in \mathcal{B}_{\varepsilon, \mathbb{O}}(\widehat{\mathbb{F}}^{\star})} \mathbb{E}^{\mathbb{F}} \left[ U(\langle w, X \rangle) \right] \right).
	$$
	Define the inner problem's value function for a fixed $w \in \mathcal{W}$ as:
	$$
	g_w^*(\varepsilon) := \inf_{\mathbb{F} \in \mathcal{B}_{\varepsilon, \mathbb{O}}(\widehat{\mathbb{F}}^{\star})} \mathbb{E}^{\mathbb{F}} \left[ U(\langle w, X \rangle) \right].
	$$
	The overall value function is then the supremum of these inner values: $g_{\rm noise}^*(\varepsilon) = \sup_{w \in \mathcal{W}} g_w^*(\varepsilon)$.

	For a fixed $w$, we analyze $g_w^*(\varepsilon)$ using Lagrangian duality. As shown in the proof of Lemma~\ref{lemma: An Equivalent Representation of Noisy-Data DRO}, strong duality holds, and we can write $g_w^*(\varepsilon)$ as the value of its dual problem:
	$$
	g_w^*(\varepsilon) = \sup_{\lambda \ge 0} \left\{ -\lambda\varepsilon + D_w(\lambda) \right\},
	$$
	where $D_w(\lambda)$ for a fixed $w$ is
	$$ 
	D_w(\lambda) := \frac{1}{N} \sum_{j=1}^{N} \inf_{x \in \mathfrak{X}} \left( U(\langle w, x \rangle) + \lambda \int_{x^\star \in \mathfrak{X}^\star} \|x^\star - \widehat{x}_j^\star\| d\mathbb{O}(x^\star \mid x) \right).
	$$
	We now establish the concavity of $D_w(\lambda)$ with respect to $\lambda$.
	To see this, we note that, for any fixed $w \in \mathcal{W}$ and $x \in \mathfrak{X}$, the expression inside the infimum is an {affine function} of $\lambda$.
	The function $\inf_{x \in \mathfrak{X}} \left( \cdot \right)$ is the pointwise infimum of a family of affine functions, which is {concave}.
	Finally, $D_w(\lambda)$ is a sum of such concave functions, which preserves concavity.
	
	Therefore, for any fixed $w \in \mathcal{W}$, function $D_w(\lambda)$ is concave in $\lambda$. This means that $g_w^*(\varepsilon)$ is found by solving a well-defined concave maximization problem. Moreover, for any fixed $\lambda \ge 0$, the objective function $-\lambda \varepsilon + D_w(\lambda)$ is affine (and therefore convex) in $\varepsilon$. Since $g_w^*(\varepsilon)$ is the pointwise supremum of a family of convex functions, it is itself a convex function of $\varepsilon$.
	Let $\lambda_w^*(\varepsilon)$ be the unique maximizer for the inner problem. The convexity of $g_w^*(\varepsilon)$ and uniqueness of the multiplier ensure that $g_w^*(\varepsilon)$ is differentiable with respect to $\varepsilon$, and its derivative is given by:
	$$
	\frac{d g_w^*(\varepsilon)}{d\varepsilon} = -\lambda_w^*(\varepsilon). \qquad(*)
	$$
	
	We now analyze the overall value function~$g_{\rm noise}^*(\varepsilon) = \sup_{w \in \mathcal{W}} g_w^*(\varepsilon)$. Let $w^*(\varepsilon)$ denote the optimal weight for a given $\varepsilon$. The assumptions of the theorem (uniqueness and continuity) ensure that the conditions for the Envelope Theorem, e.g., see \cite{sydsaeter2008further}, are satisfied for this outer problem. That is, the derivative of the value function with respect to a parameter is the partial derivative of the objective function $g_w^*(\varepsilon)$ with respect to that parameter, evaluated at the optimal choice of the decision variable $w = w^*(\varepsilon)$.
	$$
	\frac{d g_{\rm noise}^*(\varepsilon)}{d\varepsilon} = \left. \frac{d g_w^*(\varepsilon)}{d\varepsilon} \right|_{w=w^*(\varepsilon)}.
	$$
	Substituting the result from $(*)$, we get:
	$$
	\frac{d g_{\rm noise}^*(\varepsilon)}{d\varepsilon} = -\lambda_{w^*(\varepsilon)}^*(\varepsilon).
	$$
	Let us define $\lambda^*(\varepsilon) := \lambda_{w^*(\varepsilon)}^*(\varepsilon)$ as the optimal Lagrange multiplier corresponding to the globally optimal solution $(w^*(\varepsilon), \lambda^*(\varepsilon))$. This gives the final result:
	$$
	\frac{d g_{\rm noise}^*(\varepsilon)}{d\varepsilon} = -\lambda^*(\varepsilon).   \qedhere
	$$
\end{proof}

\begin{proof}[Proof of Theorem~\ref{theorem: opt values comparison}]
	For $\varepsilon > 0$, we begin by recalling the Direct DRO problem~\eqref{eq: DRO no noise} and the Noisy-Data DRO problem~\eqref{eq: convex DRO noise}:
	
	Direct DRO:
	\begin{align*}
		g^*(\varepsilon) &=
		\begin{cases}
			\displaystyle\sup_{w \in \mathcal{W}, \lambda \geq 0, a_j, \forall j} - \lambda \varepsilon +  \frac{1}{N} \sum_{j=1}^{N} a_j\\
			{\rm s.t.} \;  \displaystyle\inf_{x \in \mathfrak{X}} U( \langle w, x \rangle ) + \lambda \|x - \widehat{x}_j^\star\|  \geq a_j,\quad   j = 1, \dots, N.
		\end{cases}
	\end{align*}
	Noisy-Data DRO: 
	\begin{align*}
		g^*_{\rm noise}(\varepsilon) 
		&=
		\begin{cases}
			\displaystyle\sup_{w \in \mathcal{W}, \lambda \geq 0, a_j, \forall j} - \lambda \varepsilon +  \frac{1}{N} \sum_{j=1}^{N} a_j\\
			{\rm s.t.} \;  \displaystyle\inf_{x \in \mathfrak{X}} U( \langle w, x \rangle ) + \lambda \int_{x^\star \in \mathfrak{X}^\star}  \|x^{\star}-\widehat{x}_j^\star \| d\mathbb{O}(x^{\star} \mid x)  \geq a_j,\quad  j = 1, \dots, N.
		\end{cases}
	\end{align*} 
	To prove $g_{\rm noise}^*(\varepsilon) \geq g^*(\varepsilon)$, it suffices to show that for every feasible triplet $(w, \lambda, \{a_j\} )$ in the direct DRO model, the same triplet is feasible in the noisy-data DRO. Fix any $w \in \mathcal{W}$, $\lambda \geq 0$, and define the set of $\{a_j\}$ satisfying the constraints in each model.
	In the noisy-data DRO, for each~$j$, we require
	\[
	U( \langle w, x \rangle ) + \lambda \int_{x^\star \in \mathfrak{X}^\star}  \|x^{\star}-\widehat{x}_j^\star \| d\mathbb{O}(x^{\star} \mid x)  \geq a_j,\quad x \in \mathfrak{X}
	\]
	Observe that  for all $ x \in \mathfrak{X}$ and for each index $j$, the quantity
	\begin{align} \label{ineq: useful inequality on noisy constraint}
		\int_{x^\star \in \mathfrak{X}^\star}  \|x^{\star}-\widehat{x}_j^\star\| d\mathbb{O}(x^{\star}|x) 
		&= 
		\mathbb{E}^{\mathbb{O}(\cdot|x)} [\|X^\star-\widehat{x}_j^\star\| \mid X=x] \notag \\
		&\geq 
		\| \mathbb{E}^{\mathbb{O}(\cdot|x)} [ X^\star-\widehat{x}_j^\star \mid X=x]\| \notag \\
		&=\| \mathbb{E}^{\mathbb{O}(\cdot|x)} [ X^\star \mid X=x]-\widehat{x}_j^\star \| \notag \\
		& = \| x - \widehat{x}_j^\star \|, 
	\end{align}
	where the inequality follows from Jensen's inequality, as any norm $\|\cdot\|$ is a convex function.
	The last equation holds due to the mean-preserving assumption, which gives
	$
	\mathbb{E}^{\mathbb{O}(\cdot \mid x)} [ X^\star \mid X=x] = x
	$ 
	for all~$x \in \mathfrak{X}$.
	Thus,  with $\lambda \geq 0$, for $ x \in \mathfrak{X}$ and $ j = 1, \dots, N,$
	\[
	U( \langle w, x \rangle ) + \lambda \int_{x^\star \in \mathfrak{X}^\star}  \|x^{\star}-\widehat{x}_j^\star \| d\mathbb{O}(x^{\star} \mid x) \geq  	U( \langle w, x \rangle )  + \lambda \|x - \widehat{x}_j^\star\|.
	\]
	Taking the infimum over $x$ yields:
	\[
	\inf_{x\in \mathfrak{X}}U( \langle w, x \rangle ) + \lambda \int_{x^\star \in \mathfrak{X}^\star}  \|x^{\star}-\widehat{x}_j^\star \| d\mathbb{O}(x^{\star} \mid x) \geq  \inf_{x\in \mathfrak{X}}	U( \langle w, x \rangle )  + \lambda \|x - \widehat{x}_j^\star\|.
	\]
	Therefore, for any fixed $(w, \lambda, \{a_j\})$, the noisy-data model's constraint is less restrictive, and its feasible set is larger. Moreover, because both problems maximize the same objective over their respective feasible sets, we conclude $g^{\star}_{\mathrm{noise}}(\varepsilon) \ge g^{\star}(\varepsilon)$.
\end{proof}

\begin{proof}[Proof of Proposition~\ref{proposition: Robustness Bound under Biased Noise}]
	The proof of Theorem~\ref{theorem: opt values comparison} relies on the inequality \eqref{ineq: useful inequality on noisy constraint}. We revisit this step under biased noise. By Jensen's inequality:
	$
	\mathbb{E}^{\mathbb{O}(\cdot \mid x)} [\|X^\star-\widehat{x}_j^\star\|] \geq \|\mathbb{E}^{\mathbb{O}(\cdot \mid x)} [X^\star] - \widehat{x}_j^\star\| = \|x + b(x) - \widehat{x}_j^\star\|.
	$
	Using the reverse triangle inequality, we have:
	$$
	\|x + b(x) - \widehat{x}_j^\star\| = \|(x - \widehat{x}_j^\star) + b(x)\| \geq \|x - \widehat{x}_j^\star\| - \|b(x)\| \geq \|x - \widehat{x}_j^\star\| - \delta,
	$$
	where the last inequality holds by using $\|b(x)\| \le \delta$. Therefore, it follows that
	\begin{equation} \label{ineq:biased_bound}
		\mathbb{E}^{\mathbb{O}(\cdot \mid x)} [\|X^\star-\widehat{x}_j^\star\|] \ge \|x - \widehat{x}_j^\star\| - \delta.
	\end{equation}
	Now, let $(w^*, \lambda^*, \{a_j^*\})$ be an optimal solution for the direct problem, $g^*$. By definition, it is feasible, hence, for all $x, j$,
	$
	a_j^* \le U(\langle w^*, x \rangle) + \lambda^* \|x - \widehat{x}_j^\star\|.
	$
	Consider the solution $(w^*, \lambda^*, \{a_j' = a_j^* - \lambda^*\delta\})$. Let us check if this solution is feasible for the biased-noise problem. The feasibility condition for the biased-noise problem is $a_j' \le U(\langle w^*, x \rangle) + \lambda^* \mathbb{E}^{\mathbb{O}(\cdot \mid x)} [\|X^\star-\widehat{x}_j^\star\|]$. Substituting $a_j'$ and using the bound~\eqref{ineq:biased_bound}:
	\begin{align*}
		U(\langle w^*, x \rangle) + \lambda^* \mathbb{E}^{\mathbb{O}(\cdot|x)} [\|X^\star-\widehat{x}_j^\star\|] &\ge U(\langle w^*, x \rangle) + \lambda^* (\|x - \widehat{x}_j^\star\| - \delta) \\
		&= \left( U(\langle w^*, x \rangle) + \lambda^* \|x - \widehat{x}_j^\star\| \right) - \lambda^*\delta \\
		&\ge a_j^* - \lambda^*\delta = a_j'.
	\end{align*}
	Thus, the solution $(w^*, \lambda^*, \{a_j^* - \lambda^*\delta\})$ is feasible for the biased-noise problem. The objective value for this feasible solution is:
	$$
	-\lambda^*\varepsilon + \frac{1}{N}\sum_{j=1}^N a_j' = -\lambda^*\varepsilon + \frac{1}{N}\sum_{j=1}^N (a_j^* - \lambda^*\delta) = \left(-\lambda^*\varepsilon + \frac{1}{N}\sum_{j=1}^N a_j^*\right) - \lambda^*\delta = g^* - \lambda^*\delta.
	$$
	Since $g_{\text{noise}, \delta}^* (\varepsilon)$ is the supremum over all feasible solutions, it must be at least as large as the value for this particular feasible solution. Therefore, $g_{\text{noise}, \delta}^* (\varepsilon) \geq g^*(\varepsilon) - \lambda^* \delta$. 
\end{proof}

\begin{proof}[Proof of Lemma~\ref{lemma: Dominance of the Price of Ambiguity}] 
	To establish a proof of part~$(i)$, note that $\mathtt{POA}_{\text{noise}}(\varepsilon) =  \frac{\text{SYSTEM} - g_{\rm noise}^*(\varepsilon)}{\text{SYSTEM}}$.
	By Theorem~\ref{theorem: Sensitivity of the Value Function}, the value function $g_{\rm noise}^*(\varepsilon)$ is differentiable with respect to $\varepsilon$, and its derivative is given by $\lambda^*(\varepsilon)$.
	Hence $\mathtt{POA}_{\text{noise}}(\varepsilon)$ is differentiable, and differentiating it with respect to $\varepsilon$ yields:
	\begin{align*}
		\frac{d  \mathtt{POA}_{\rm noise}(\varepsilon)}{d\varepsilon} &=
		-\frac{d g_{\rm noise}^*(\varepsilon)}{d\varepsilon} \cdot \frac{1}{\text{SYSTEM}} \\
		&= \lambda^*(\varepsilon)\cdot \frac{1}{\text{SYSTEM}}.
	\end{align*} 
	Since $\lambda^*(\varepsilon) \geq 0$ and $\text{SYSTEM} > 0$, the derivative is non-negative, and hence $\mathtt{POA}_{\text{noise}}(\varepsilon)$ is nondecreasing in $\varepsilon$.
	A similar argument applies to the direct DRO case, differentiating the value function $g^*(\varepsilon)$ in the same way shows that its derivative is non-negative, and thus $\mathtt{POA}$ is also nondecreasing.
	
	To prove part $(ii)$, it remains to show that the required dominance of $\mathtt{POA}$ holds. Indeed,~for $\varepsilon > 0$, using Theorem~\ref{theorem: opt values comparison}, since $g_{\rm noise}^*(\varepsilon) \geq g^*(\varepsilon)$ for all~$\varepsilon > 0$, the desired inequality holds:
	\begin{align*}
		&\mathtt{POA}(\varepsilon) \geq \mathtt{POA}_{\rm noise}(\varepsilon). \qedhere
	\end{align*}
\end{proof}

\section{Noise Models and Fairness}
\subsection{Fairness Model Formulations in Section~\ref{section: Illustrative Examples}} \label{sec: Fairness Model Formulations}  

Applying the $\alpha$-fairness utility function in Definition~\ref{definition: modified alpha fairness} to the two models $\texttt{DRO}_{\rm {noise}}$ and $\texttt{DRO}$, presented in Lemma~\ref{lemma: An Equivalent Representation of Noisy-Data DRO} and Remark~\ref{remark: noise free DRO}, we obtain the following equivalent formulations, which are used in our empirical~study.

$(i)$
The $\texttt{DRO}_{\rm {noise}}$ model:
\begin{align*}
	&\displaystyle\sup_{w \in \mathcal{W}, \, \lambda \geq 0, \, a_j, \, z_{x^\star, j}, \forall j, \forall x^\star} - \lambda \varepsilon +  \frac{1}{N} \sum_{j=1}^{N} a_j\\
	&{\rm s.t.} \;  
	\displaystyle\min_{x \in \mathfrak{X}}  \sum_{i=1}^n \frac{1}{1-\alpha}\left(  (w_i x_i + 1 )^{1 - \alpha} -1 \right) + \sum_{x^\star \in \mathfrak{X}^\star}  \langle z_{x^\star, j}, x^\star - \widehat{x}_j^\star \rangle  \mathbb{O}(x^\star \mid x)  \geq a_j,\quad  j = 1, \dots, N ; \notag \\
	&\quad\;\; \|z_{x^\star, j} \|_* \leq \lambda, \quad x^\star \in \mathfrak{X}^\star,\; j= 1, \dots, N. \notag
\end{align*}

$(ii)$
The $\texttt{DRO}$ model:
\begin{align*}
	&\displaystyle\sup_{w \in \mathcal{W}, \, \lambda \geq 0, \, a_j, \, z_{j}, \forall j} - \lambda \varepsilon +  \frac{1}{N} \sum_{j=1}^{N} a_j\\
	&{\rm s.t.} \;  
	\displaystyle \min_{x \in \mathfrak{X}}  \sum_{i=1}^n \frac{1}{1-\alpha}\left(  (w_i x_i  + 1 )^{1 - \alpha} -1 \right) +  \langle z_{j}, x -\widehat{x}_j^\star \rangle \geq a_j,\quad  j = 1, \dots, N;\notag \\
	&\quad\;\; \|z_{j} \|_* \leq \lambda, \quad j= 1, \dots, N. \notag
\end{align*}

\subsection{Noise Distribution Models in Section~\ref{section: Illustrative Examples}}\label{appendix: noise distribution}

Below, we provide the probability mass function (PMF) $\mathbb{P}(E=e)$ and support sets $\mathfrak{E}$ for the additive noise models tested in Section~\ref{section: Illustrative Examples}, along with their parameter settings.
In our illustrative example, data are rescaled to lie in~$[0, 1]$.
Let $\text{diam}(\mathfrak{X}) := \max_{x, x' \in \mathfrak{X}}\| x - x' \|_1 $ and define $d := \frac{\text{diam}(\mathfrak{X})}{n}$. 
Since the data are rescaled, we have $d=1$.
We now specify the support and distribution of the noise $E$ under different models.

$(i)$ Discrete truncated normal with parameters $\mu$, $\sigma$, $a$, and $b$:
The support is a discretized grid of points within the set $\mathfrak{E} = \{ e \in \mathbb{R}^n : a \leq e_i \leq b, \ i = 1, \ldots, n\} $, where  $-1 < a < b < \infty$. The distribution is
$$
\mathbb{O}(x^\star \mid x) = \mathbb{P} (E = e) = \phi_{\rm trunc}(e; \mu, \sigma, a, b),
$$ 
where $\phi_{\rm trunc}(e; \mu, \sigma, a, b)$ is the truncated normal probability mass function with mean $\mu$, standard deviation $\sigma$, and truncation bounds $a$, $b$.
In our example, we set $\mu = 0$, $\sigma$ equal to the standard deviation of the data, and $a = -0.01$, $b= 0.01$.

$(ii)$ Discrete softmax with parameters $(a, b)$:
The support set is $\mathfrak{E} = \{ e \in \mathbb{R}^n : a  \leq e_i \leq b, \ i = 1, \ldots, n\} $, where $-1 < a < b < \infty$. The distribution is 
$$
\mathbb{O}(x^\star \mid x) = \mathbb{P} (E = e) = \frac{\exp \left(\frac{\|e\|_1}{\text{diam}(\mathfrak{X})}\right)}{\sum_{e' \in \mathfrak{E}} \exp \left( \frac{\|e'\|_1}{\text{diam}(\mathfrak{X})} \right)}
.$$ 
In our example, we use $a=-0.01$ and $b=0.01$

$(iii)$ Bernoulli with parameters $(p, a)$:
The support set is $\mathfrak{E} = \{ e \in \mathbb{R}^n : e_i \in \{0 , \ a \}, \ i = 1, \ldots, n\} $, where~$a \geq 0$.
The distribution is
$$
\mathbb{O}(x^\star \mid x)  = \mathbb{P} (E = e) = \prod_{i = 1}^n p^{\mathbb{1}_{ \{e_i = 0.01\} }} (1- p)^{\mathbb{1}_{ \{e_i = 0\} }} ,
$$
where $\mathbb{1}_{\{\cdot\}}$ is the indicator function, and $0 \leq p \leq 1$ is probability of that the noise $e_i$ takes the value $a$.
In our example, we set $p=0.5$ and $a = 0.01$.

$(iv)$ Binomial with parameters $(p, m, a)$:
The support is $\mathfrak{E} = \{ e \in \mathbb{R}^n : e_i = k_i a, \ k_i = 0, 1, \cdots , m, \ i = 1, \ldots, n \}$, where $0 \leq p \leq 1$ is the success probability, $m \in \mathbb{N}$ is the number of trials, and $a \geq 0$ is the step size. The distribution is
$$
\mathbb{O}(x^\star \mid x) = \prod_{i = 1}^n \left( \begin{array}{c} m \\ k_i \end{array} \right) p^{k_i} (1 - p)^{m - k_i}
, \quad k_i = 0, 1, \cdots , m.
$$ 
We use $p=0.5$, $m = 2$, and $a = 0.001$ in our example.

$(v)$ Poisson with parameters $(\lambda, a)$:
The support is $\mathfrak{E} = \{ e \in \mathbb{R}^n : e_i = k_i a, \ k_i \in \mathbb{N}, \ i = 1, \ldots, n \}$, where $\lambda > 0$ is the Poisson rate parameter and $a$ is the step size. The distribution is
$$
\mathbb{O}(x^\star \mid x) = \prod_{i = 1}^n \frac{\exp(-\lambda) \lambda^{k_i}}{k_i!}, \quad k_i \in \mathbb{N}
.$$ 
In our example, we use $\lambda = 0.1$ and $a = 0.01$.

\end{document}